%% file: 0preprint.tex
\documentclass[12pt]{amsart}
\usepackage{amsmath,amsfonts,amsthm,amssymb,nameref}
\usepackage[top=4cm, bottom=3cm, left=3.9cm, right=3.9cm]{geometry} 
\usepackage{tikz} 
\usetikzlibrary{calc} 
\usepackage[pdftex]{hyperref}
\hypersetup{colorlinks,
citecolor=black,
filecolor=black,
linkcolor=black,
urlcolor=black}
\usepackage[all]{hypcap}
\newcommand{\N}	{\mathbb N}
\newcommand{\Z}	{\mathbb Z}
\newcommand{\R}	{\mathbb R}

\newcommand{\Cay}	{\operatorname{Cay}}
\newcommand{\diam}	{\operatorname{diam}}
\newcommand{\Con}	{\operatorname{Con}}
\newcommand{\red}	{\operatorname{red}}
\newcommand{\Area}	{\operatorname{Area}}
\newtheorem{thm}{Theorem}[section] 
\newtheorem{prop}[thm]{Proposition}
\newtheorem{lem}[thm]{Lemma} 
\newtheorem{cor}[thm]{Corollary}
\newtheorem*{thm*}{Theorem} 
\newtheorem*{prop*}{Proposition}
\newtheorem*{lem*}{Lemma} 
\newtheorem*{cor*}{Corollary}
\theoremstyle{definition}
\newtheorem{defi}[thm]{Definition}
\newtheorem{example}[thm]{Example}
\newtheorem{remark}[thm]{Remark}
\newtheorem*{defi*}{Definition}
\newtheorem*{example*}{Example}
\newtheorem*{remark*}{Remark}
\begin{document}
\title[Graphical $C(6)$ and $C(7)$ small cancellation groups]{Groups with graphical $C(6)$ and $C(7)$ small cancellation presentations}
\author{Dominik Gruber}
 \thanks{The work is supported by the ERC grant of Prof. Goulnara Arzhantseva ``ANALYTIC" no. 259527.}
 \address{University of Vienna, Department of Mathematics, Oskar-Morgenstern-Platz 1, 1090 Wien, Austria}
 \email{dominik.gruber@univie.ac.at}
\begin{abstract}
We extend fundamental results of small cancellation theory to groups whose presentations satisfy the generalizations of the classical $C(6)$ and $C(7)$ conditions in \emph{graphical} small cancellation theory. Using these graphical small cancellation conditions, we construct lacunary hyperbolic groups and groups that coarsely contain prescribed infinite sequences of finite graphs. We prove that groups given by (possibly infinite) graphical $C(7)$ presentations contain non-abelian free subgroups.
\end{abstract}
\maketitle
\input{1introduction2}
\input{2basics}
\input{3asphericity}

\input{4free_subgroups}
\input{5coarse_embedding2}
\input{6lacunary_hyperbolicity}

\input{references}
\end{document}

%% file: 1introduction2.tex
\section{Introduction}

In his influential paper ``Random walk in random groups" \cite{Gr}, Gromov introduced small cancellation theory for labelled graphs as a far-reaching generalization of classical small cancellation theory. 
A main feature of this theory is that it allows constructions of finitely generated groups that contain prescribed subgraphs in an appropriate metric sense. Gromov applied this observation to construct \emph{Gromov's monster}, a group that coarsely contains an expander graph. (For details, see \cite{AD}.)

Subsequently, Ollivier gave details on another theorem of Gromov that extends results of classical $C'(\frac{1}{6})$ small cancellation theory to its analogue in graphical small cancellation theory \cite{Oll}.

In the present paper, we generalize Ollivier's proof to the graphical analogues of the classical $C(6)$ and $C(7)$ conditions. The method of proof we present is purely combinatorial; most of our statements rely solely on Lyndon's curvature formulas \cite{LS}, which are consequences of the formula for the Euler characteristic of planar 2-complexes, and their applications to suitable van Kampen diagrams. 

We first use our method to prove generalizations of fundamental results about classical $C(6)$ and $C(7)$ groups concerning Dehn functions and asphericity. These results have been stated before \cite{AD,Gr,Oll}, but proofs have never been published.

We then establish new results about groups given by non-metric graphical small cancellation presentations: Any group given by a (possibly infinite) graphical $C(7)$ presentation, see Definition \ref{defi:graphical1}, contains a non-abelian free subgroup, see Theorem \ref{thm:free}. This in particular implies non-amenability and exponential growth of the group.

We moreover prove that the $Gr(6)$ condition, see Definition \ref{defi:graphical2}, can be used to construct groups that coarsely contain prescribed infinite sequences of finite graphs, see Theorem \ref{thm:coarse_embedding}. In the view of applications to the Baum-Connes conjecture \cite{HLS}, this is of particular interest when considering an expander graph: If an expander graph $X:=(X_n)_{n\in\N}$ admits a labelling such that the disjoint union $\sqcup_{n\in\N}X_n$ satisfies the $Gr(6)$ condition, then the group defined by the labelled graph $X$ coarsely contains $X$. Thus, such a labelling would yield a new construction of Gromov's monster.

Continuing our investigations of groups defined by infinite sequences of finite graphs, we show that if the $Gr(7)$ condition, see Definition \ref{defi:graphical2}, is satisfied, lacunary hyperbolic groups can be constructed by passing to appropriate subsequences of defining graphs, see Proposition \ref{prop:7_lacunary}. If, moreover, the metric $Gr'(\frac{1}{6})$ condition is satisfied, we are able to give combinatorial conditions on the graphs that yield lacunary hyperbolicity, see Proposition \ref{prop:1/6_lacunary}.

The construction of Gromov's monster inspired further constructions of groups with extreme properties through graphical small cancellation theory \cite{AD, OW, Sil}. Common features of all these constructions are their use of \emph{metric} graphical small cancellation conditions and of \emph{random} labellings of graphs. Therefore, the group presentations obtained in this manner are non-explicit. Turning them into explicit constructions has been unattainable so far. 

We believe that the weaker \emph{non-metric} small cancellation conditions
we present are more accessible to making explicit constructions. Our
theorems show that, in some respects, they yield results as good
as the metric ones.

\subsection{Graphical small cancellation conditions} 

Before stating the graphical small cancellation conditions, we explain how a group is constructed from a labelled graph. This idea first appeared in \cite{RS}. Throughout this paper, $S$ denotes a finite set (our \emph{alphabet}). 

Let $\Gamma$ be a graph. A \emph{labelling} of $\Gamma$ with letters from $S$ is a choice of orientation on each edge of $\Gamma$ and a map assigning to each edge of $\Gamma$ an element of $S$ (the \emph{label}). A map of labelled graphs is a simplicial map that preserves the labelling. A labelling is called \emph{reduced} if at every vertex $v$, any two oriented edges starting at $v$ have distinct labels and any two oriented edges ending at $v$ have distinct labels. 

Let $M(S)$ denote the free monoid on $S\sqcup S^{-1}$. The labelling of $\Gamma$ induces a map
$$\ell:\{\text{paths on }\Gamma\}\to M(S),$$
given as follows: For a path $p$, $\ell(p)$ is obtained by reading the labels of the edges in $p$ starting from the initial vertex of $p$,
and each letter is given exponent 1 if the corresponding edge is traversed in its direction and exponent $-1$ if it is traversed in the opposite direction.

Assume $\Gamma$ is connected and fix a base vertex $v$ in $\Gamma$. Then $\ell$ induces a homomorphism:
$$\ell_*:\pi_1(\Gamma,v)\to F(S),$$
where $F(S)$ denotes the free group on $S$. Similarly, if $\Gamma$ is the disjoint union of a set of connected components with base points $\{(\Gamma_i,v_i)|i\in I\}$, we obtain a homomorphism $\ell_*:*_{i\in I}\pi_1(\Gamma_i,v_i)\to F(S)$. 

\begin{defi}\label{defi:groupdefinedbygraph} For a labelled graph $\Gamma$, we define
$$G(\Gamma):=F(S)/\langle\langle \operatorname {im} \ell_* \rangle\rangle,$$ 
the \emph{group defined by $\Gamma$}. (Here $\langle\langle - \rangle\rangle$ denotes the normal subgroup generated by $-$.) Thus $G(\Gamma)$ is given by the presentation $\langle S|R\rangle$, where $R$ is the set of words read on closed paths in $\Gamma$.
\end{defi}

Let $\Gamma$ be the disjoint union of its connected components $\Gamma_i$ for $i\in I$. Given vertices $v_i\in \Gamma_i$ and group elements $g_i\in G(\Gamma)$, there is a unique map of labelled graphs $\Gamma\to\Cay(G(\Gamma),S)$ that maps each $v_i$ to $g_i$. Any normal subgroup $N$ of $F(S)$ with the property that there exists a map of labelled graphs $\Gamma\to \Cay(F(S)/N,S)$ contains $\langle\langle \operatorname {im} \ell_* \rangle\rangle$, since any closed path in $\Gamma$ has to be mapped to a closed path in $\Cay(F(S)/N,S)$. Thus, we can think of $G(\Gamma)$ as the largest quotient $G$ of $F(S)$ such that $\Gamma$ maps to $\Cay(G,S)$.

\vspace{8pt}

We now state the graphical small cancellation conditions. The graphical $C(n)$ and $C'(\lambda)$ conditions are based on \cite{Oll} and \cite[Section 2.2]{AD}, who both slightly modified Gromov's original definition \cite[Section 2]{Gr}. Gromov's graphical small cancellation conditions will be denoted by $Gr(n)$ and $Gr'(\lambda)$. We call a closed path on a graph \emph{nontrivial} if it is not 0-homotopic.
\begin{defi}\label{defi:graphical1}
Let $\Gamma$ be a labelled graph.  A \emph{piece} on $\Gamma$ is a labelled path $p$ (considered as a labelled graph) for which there are two distinct maps of labelled graphs $p\to\Gamma$. Let $n\in\N$ and $\lambda>0$.
\begin{itemize}
 \item $\Gamma$ satisfies the $C(n)$ \emph{condition} if the labelling is reduced and no nontrivial closed path is the concatenation of fewer than $n$ pieces.
 \item $\Gamma$ satisfies the $C'(\lambda)$ \emph{condition} if the labelling is reduced and for any simple closed path $\gamma$, any piece $p$ that is a subpath of $\gamma$ satisfies $|p|<\lambda |\gamma|$.
\end{itemize}
\end{defi}

\begin{defi}\label{defi:graphical2}
Two maps of labelled graphs $\phi_1,\phi_2:\Gamma'\to \Gamma$ are \emph{essentially equal} if there is an automorphism of labelled graphs $\psi:\Gamma\to\Gamma$ with $\phi_2=\psi\circ\phi_1$. Otherwise they are \emph{essentially distinct}. An \emph{essential piece} on $\Gamma$ is a labelled path $p$ for with there exist two essentially distinct maps of labelled graphs $p\to\Gamma$.
\begin{itemize}
 \item $\Gamma$ satisfies the $Gr(n)$ \emph{condition} if the labelling is reduced and no nontrivial closed path is the concatenation of fewer than $n$ essential pieces.
 \item $\Gamma$ satisfies the $Gr'(\lambda)$ \emph{condition} if the labelling is reduced and for any simple closed path $\gamma$, any essential piece $p$ that is a subpath of $\gamma$ satisfies $|p|<\lambda |\gamma|$.
\end{itemize}
\end{defi}

Observe that the graphical small cancellation conditions are generalizations of the classical small cancellation conditions: For any classical $C(n)$ or $C'(\lambda)$ presentation $\langle S|R\rangle$, we can construct a graph as the disjoint union of cycle graphs, each labelled by (the class of cyclic conjugates and their inverses) of a relator of the presentation. If $\langle S|R\rangle$ satisfies the classical $C(n)$ or $C'(\lambda)$ condition, the constructed graph satisfies the $Gr(n)$ or $Gr'(\lambda)$ condition, respectively. If no relator is a proper power, the graph satisfies the $C(n)$ or $C'(\lambda)$ condition, respectively.

\subsection{Statement of results}

We first state those results which have been expected based on \cite{Gr}, \cite{Oll} and \cite{AD}. They are generalizations of fundamental facts about classical $C(6)$ and $C(7)$ presentations. Let $S$ be a finite set.

\begin{thm*}[cf.\ Theorem \ref{thm:linear_isoperimetry}]
 Let $\Gamma$ be a $Gr(7)$-labelled graph. Let $R$ be the set of words read on all simple cycles of $\Gamma$. Then the presentation $\langle S|R\rangle$ satisfies the linear isoperimetric inequality:
$$\Area_R(w)\leq 8|w|.$$
In particular, if $\Gamma$ is finite, then the group $G(\Gamma)$ is Gromov-hyperbolic.
\end{thm*}

\begin{thm*}[cf.\ Theorem \ref{thm:quadratic_isoperimetry}]
 Let $\Gamma$ be a $Gr(6)$-labelled graph. Let $R$ be the set of words read on all simple cycles of $\Gamma$. Then the presentation $\langle S|R\rangle$ satisfies the quadratic isoperimetric inequality:
$$\Area_R(w)\leq 3|w|^2.$$
\end{thm*}

\begin{thm*}[cf. Theorem \ref{thm:asphericity}] Let $\Gamma$ be a $C(6)$-labelled graph. Let $R$ be the set of cyclic reductions of words read on free generating sets of the fundamental groups of the connected components of $\Gamma$. Then the presentation complex associated to $\langle S|R\rangle$ is aspherical. The group $G(\Gamma)$ defined by $\Gamma$ has an at most 2-dimensional $K(G(\Gamma),1)$ space and hence cohomological dimension at most 2. Therefore, $G(\Gamma)$ is torsion-free. 
\end{thm*}

We next state our original results. The following theorem in particular shows that only non-amenable groups arise from graphical $C(7)$ presentations. While for finitely presented graphical $C(7)$ groups the theorem follows from the fact that non-elementary hyperbolic groups have non-abelian free subgroups \cite[$5.3.\mathrm{C}_1$]{Grhyp}, it is new for infinitely presented graphical $C(7)$ groups.

\begin{thm*}[cf. Theorem \ref{thm:free}]
 Let $\Gamma$ be a $C(7)$-labelled graph. Then $G(\Gamma)$ contains a non-abelian free subgroup unless it is trivial or infinite cyclic.
\end{thm*}

In fact, we explain in Remark \ref{rem:obvious} that $G(\Gamma)$ is trivial or infinite cyclic if and only if some easily checkable conditions on the graph hold. We also show a version of this theorem applicable to classical small cancellation presentations:

\begin{thm*}[cf. Theorem \ref{thm:ess-free}]
 Let $\Gamma$ be a $Gr(7)$-labelled graph with infinitely many pairwise non-isomorphic connected components with nontrivial fundamental groups. Then $G(\Gamma)$ contains a non-abelian free subgroup.
\end{thm*}

\begin{cor*}[cf. Corollary \ref{cor:classical-free}]
 Let $G$ be a group with an infinite classical $C(7)$ presentation. Then $G$ contains a non-abelian free subgroup.
\end{cor*}

We next consider how the defining graph $\Gamma$ of a group embeds into the Cayley graph $\Cay(G(\Gamma),S)$ as a metric space. Let $(\Gamma_n)_{n\in\N}$ be a sequence of connected finite graphs and let $\Gamma:=\bigsqcup_{n\in\N}\Gamma_n$ be their disjoint union. We endow $\Gamma$ with a metric that coincides with the graph metric on each connected component such that $d(\Gamma_ {a_n},\Gamma_{b_n})\to\infty$ as $a_n+b_n\to \infty$ assuming $a_n\neq b_n$ for almost all $n$. We call the resulting metric space the \emph{coarse union} of the $\Gamma_n$. A \emph{coarse embedding} is a map $f:X\rightarrow Y$, where $X$ and $Y$ are metric spaces, such that for all sequences $(x_n,y_n)_{n\in\N}$ in $X\times X$ we have $d(x_n,y_n)\to\infty\Leftrightarrow d(f(x_n),f(y_n))\to\infty$.

The following theorem shows that whenever we find a labelling of an infinite sequence of finite graphs such that their disjoint union satisfies the $Gr(6)$ condition, we can construct a finitely generated group that coarsely contains the coarse union of these graphs:

\begin{thm*}[cf.\ Theorem \ref{thm:coarse_embedding}] Let $(\Gamma_n)_{n\in\N}$ be a sequence of finite, connected graphs such that $\Gamma:=\bigsqcup_{n\in\N}\Gamma_n$ is $Gr(6)$-labelled and such that $|\Gamma_n|$ is unbounded. Then the coarse union $\Gamma$ embeds coarsely into $\Cay(G(\Gamma),S)$.
\end{thm*}

This map is also injective as stated by Gromov \cite{Gr} and proven in Lemma \ref{lem:embedding_injective}. 

A finitely generated group is called \emph{lacunary hyperbolic} if one of its asymptotic cones is an $\R$-tree. The following proposition shows that the $Gr(7)$ condition can be used to construct lacunary hyperbolic groups:

\begin{prop*}[cf.\ Proposition \ref{prop:7_lacunary}]
 Let $(\Gamma_n)_{n\in\N}$ be a sequence of finite, connected graphs such that their disjoint union is $Gr(7)$-labelled. Then there exists an infinite subsequence of graphs $(\Gamma_{k_n})_{n\in\N}$ such that $G(\bigsqcup_{n\in\N}\Gamma_{k_n})$ is lacunary hyperbolic.
\end{prop*}

Moreover, we apply results of \cite{Oll} and an argument given in 
\cite[Proposition 3.12]{OOS}, where the question when a group given by an infinite classical $C'(\frac{1}{6})$ presentation is lacunary hyperbolic is answered, to show:

\begin{prop*}[cf.\ Proposition \ref{prop:1/6_lacunary}] Let $(\Gamma_n)_{n\in\N}$ be a sequence of finite, connected graphs such that
 \begin{itemize}
  \item $\Gamma:=\sqcup_{n\in\N}\Gamma_n$ is $Gr'(\frac{1}{6})$-labelled and
  \item $\Delta(\Gamma_n)=O(g(\Gamma_n))$,
 \end{itemize}
where $\Delta(\Gamma_n)$ denotes the diameter and $g(\Gamma_n)$ denotes the girth of each graph. (If $\Gamma_n$ is a tree, set $g(\Gamma_n)=0$.)
Then $G(\Gamma)$ is lacunary hyperbolic if and only if the set of girths $L:=\{g(\Gamma_n)|n\in\N\}$ is sparse, i.e. for all $K\in \R^+$ there exists $a \in \R^+$ such that $[a,aK]\cap L=\emptyset$.
\end{prop*}

\subsection*{Acknowledgements} The author would like to thank his advisor, Goulnara Arzhantseva, for suggesting the topic, sharing her expertise and giving encouragement to write this paper and his colleagues from the Geometric group theory research group at the University of Vienna for helpful discussions, comments and corrections. The author also appreciates the anonymous referee's constructive remarks.

%% file: 2basics.tex
\section{Extending classical small cancellation theory}\label{section:preliminaries}

\subsection{Examples and first observations}\label{subsection:p1}
Following the definitions of the graphical small cancellation conditions given in the introduction, we provide two examples and make some basic observations about pieces and the graphical small cancellation conditions. 

\begin{example}
 
Let $S=\{a,b,c\}$, and let $\Gamma$ be as in Figure \ref{figure:C(6)-example}. The group $G(\Gamma)$ is given by the presentation $$\langle a,b,c|a^2c^{-1}b^{-2}a^{-1}b^{-1},a^2b^{-1}c^{-2}a^{-1}c^{-1}\rangle.$$ The graph $\Gamma$ satisfies the $C'(\frac{1}{6})$ condition: The labelling is reduced, any piece that is subpath of a simple cycle has length at most 1, and any cycle has length at least 7. Note that the presentation does not satisfy the classical $C'(\frac{1}{6})$ condition. Also note that this is a maximal graphical $C'(\frac{1}{6})$ presentation in the alphabet $S$, i.e.\ one cannot add more $S$-labelled cycles to the graph retaining the $C'(\frac{1}{6})$ condition, since any such cycle would produce new pieces of length 2 on $\Gamma$ that would lead to a violation of the $C'(\frac{1}{6})$ condition.

\end{example}

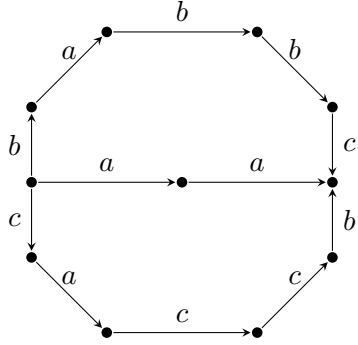
\begin{figure*}[ht]
\begin{tikzpicture}[>=stealth,shorten <=2.5pt, shorten >=2.5pt]
\node[coordinate] at (0,1) (X1) {};
\node[coordinate] at (1,2) (X2) {};
\node[coordinate] at (3,2) (X3) {};
\node[coordinate] at (4,1) (X4) {};

\node[coordinate] at (0,0) (Y1) {};
\node[coordinate] at (2,0) (Y2) {};
\node[coordinate] at (4,0) (Y3) {};

\node[coordinate] at (0,-1) (Z1) {};
\node[coordinate] at (1,-2) (Z2) {};
\node[coordinate] at (3,-2) (Z3) {};
\node[coordinate] at (4,-1) (Z4) {};

\fill (X1) circle (2pt);\fill (X2) circle (2pt);\fill (X3) circle (2pt);\fill (X4) circle (2pt);
\fill (Y1) circle (2pt);\fill (Y2) circle (2pt);\fill (Y3) circle (2pt);
\fill (Z1) circle (2pt);\fill (Z2) circle (2pt);\fill (Z3) circle (2pt);\fill (Z4) circle (2pt);

\draw[->] (X1) to node [above] {\small $a$} (X2);
\draw[->] (X2) to node [above] {\small $b$} (X3);
\draw[->] (X3) to node [above] {\small $b$} (X4);

\draw[->] (Y1) to node [left] {\small $b$} (X1);\draw[->] (X4) to node [right] {\small $c$} (Y3);
\draw[->] (Y1) to node [above] {\small $a$} (Y2);\draw[->] (Y2) to node [above] {\small $a$} (Y3);
\draw[->] (Y1) to node [left] {\small $c$} (Z1);\draw[->] (Z4) to node [right] {\small $b$} (Y3);

\draw[->] (Z1) to node [above] {\small $a$} (Z2);
\draw[->] (Z2) to node [above] {\small $c$} (Z3);
\draw[->] (Z3) to node [above] {\small $c$} (Z4);
\end{tikzpicture}
\caption{The $C'(\frac{1}{6})$-labelled graph $\Gamma$.}
\label{figure:C(6)-example}
\end{figure*}

\begin{example}\label{ex:cayley}
 Let $G$ be a group generated by a finite set $S$. Then the Cayley graph $\Cay(G,S)=:\Gamma$ has a natural reduced labelling, and we have $G=G(\Gamma)$. Note that a path on a graph with a reduced labelling is uniquely determined by its starting vertex and its label. The group of labelled-graph-automorphisms of $\Gamma$ acts transitively on the vertex set of $\Gamma$. Therefore, whenever a labelled path has two maps to $\Gamma$, they are essentially equal. Thus $\Gamma$ has no essential pieces and satisfies any $Gr(n)$ or $Gr(\lambda)$ condition. Unless the group is free on $S$, or trivial, it will not satisfy $C(2)$ or $C'(1)$, i.e.\ no nontrivial $Gr(n)$ or $Gr'(\lambda)$ condition, since there will be two distinct cycles with the same label and thus a cycle made up of one piece.
\end{example}

Observe that any essential piece is a piece. Thus the $C(n)$ condition implies the $Gr(n)$ condition, and the $C'(\lambda)$ condition implies the $Gr'(\lambda)$ condition. The treatment of the conditions from Definition \ref{defi:graphical1} and those from Definition \ref{defi:graphical2} is very similar; in fact all results of Subsections \ref{subsection:p1}-\ref{subsection:p3} hold for both conditions and for pieces and essential pieces alike. We will therefore mostly restrict ourselves to stating them for the $C(n)$ or $C'(\lambda)$ conditions and for pieces only. The only difference in statements and proofs is that ``piece" has to be replaced by ``essential piece", ``unique" by ``essentially unique", ``equal" by ``essentially equal", etc.

\vspace{8pt}

\emph{Backtracking} in a path on a graph $\Gamma$ means that an edge of $\Gamma$ is traversed two consecutive times in opposite directions. A path in a graph is \emph{reduced} if there is no backtracking. The reduction of a path is obtained by removing any backtracking. We have the following observations for a piece $p$:
\begin{itemize}
 \item The reduction of $p$ is a piece.
 \item Every subpath of $p$ is a piece.
 \item The inverse path of $p$ is a piece. 
\end{itemize}

\begin{lem} Let $\Gamma$ be a labelled graph, $n\in\N$ and $\lambda>0$. Then the following are equivalent:
\begin{itemize}
\item[i)] $\Gamma$ satisfies $C(n)$.
\item[ii)] No reduced cycle is the concatenation of fewer than $n$ pieces.
\item [iii)] No simple cycle is the concatenation of fewer than $n$ pieces.
\end{itemize}
Moreover, if $\Gamma$ satisfies $C'(\lambda)$, then $\Gamma$ satisfies $C(\lfloor \frac{1}{\lambda} \rfloor +1)$.
\end{lem}
\begin{proof}
 The implications $\text{i)}\Rightarrow \text{ii)}\Rightarrow \text{iii)}$ are obvious. Now suppose iii) holds. Let $\gamma$ be a nontrivial cycle on $\Gamma$, and suppose $\gamma=p_1p_2\ldots p_k$ where $p_i$ is a piece for each $i$. Then the reduction of $\gamma$ may be written as $q_1q_2\ldots q_l$ where each $q_i$ is a subpath of the reduction of a $p_j$, and $l\leq k$. Let $\gamma'$ be a subpath of $\gamma$ that is a simple cycle. Then $\gamma'=r_1r_2\ldots r_m$, where each $r_i$ is a subpath of some $q_j$, and $m\leq l$. By assumption, we have $m\geq n$ and hence $k\geq l\geq m\geq n$.
 
The last statement follows from the implication $\text{iii)}\Rightarrow \text{i)}$.
\end{proof}

\subsection{Diagrams}\label{subsection:p2}

We briefly recall tools and results of classical small cancellation theory, which we will use to present graphical small cancellation theory. For an introduction to classical small cancellation theory see \cite[Chapter V]{LS}.

\vspace{8pt}

A \emph{singular disk diagram} $D$ in the alphabet $S$ is a finite, simply connected, 2-dimensional CW-complex embedded into the plane with the following additional data:
\begin{itemize}
 \item Each 1-cell is oriented and labelled by an element of $S$.
 \item A base vertex $v$ in the topological boundary of $\R^2\setminus D$ is specified. (We also denote $(D,v)$. We omit mentioning the base vertex if this is not necessary.)
\end{itemize}
A \emph{simple} disk diagram is a singular disk diagram that is homeomorphic to a disk. Equivalently, a simple disk diagram is singular disk diagram without cut-points.

The \emph{boundary path} $\partial D$ of a singular disk diagram $D$ is the closed simplicial path along the topological boundary of $\R^2\setminus D$ from the base vertex in counterclockwise direction. The \emph{label} of $\partial D$ (or \emph{boundary label} of $D$) is the element of the free monoid on $S\sqcup S^{-1}$ obtained by reading the labels of the edges of $\partial D$,
where each letter is given exponent 1 if the corresponding edge is traversed in its direction and exponent $-1$ if it is traversed in the opposite direction.

We also introduce spherical diagrams following \cite{CH}. A spherical complex is a set of 2-spheres embedded into $\R^3$ connected by simple curves such that no sphere contains the other, such that the intersection of two spheres consists of at most one point and such that the whole complex is simply connected. (Recall that a curve is simple if it has no self-intersections.)

A \emph{spherical diagram} in the alphabet $S$ is finite 2-complex tessellating a spherical complex, where each 1-cell is oriented and labelled by an element of $S$. A \emph{simple} spherical diagram is a spherical diagram tessellating a single 2-sphere. It is our convention that a spherical diagram has empty boundary.

The word \emph{diagram} will refer to a singular disk diagram or a spherical diagram. The \emph{area} of a diagram is the number of its 2-cells.

A \emph{face} $f$ of a diagram is the image of a closed 2-cell $b$ under its characteristic map. The \emph{boundary path} $\partial f$ of $f$ is the simplicial path obtained as the image of the simple closed path along the boundary of $b$. The \emph{label} of $\partial f$ (or \emph{boundary label} of $f$) is obtained by reading the labels of the edges of $\partial f$. (Here choices of basepoints and orientations are necessary. If these are not specified, the boundary path and boundary label are still unique up to inversion and cyclic shifts, which will often be sufficient for our considerations.)

For the finite set $S$, let $M(S)$ denote the free monoid on $S\sqcup S^{-1}$. For a set of words $R\subset M(S)$, a \emph{diagram over $R$} is a diagram $D$ such that for each face $f$ of $D$ there exist an orientation and a basepoint for  $\partial f$ such that the label of $\partial f$ lies in $R$. Given a labelled graph $\Gamma$, a \emph{diagram over $\Gamma$} is a diagram over the set of words read on nontrivial cycles on $\Gamma$.

Let $w\in M(S)$. A \emph{diagram for $w$ over $R$} (respectively \emph{over $\Gamma$}) is a diagram over $R$ (respectively $\Gamma$) whose boundary label is $w$.  A \emph{minimal} diagram for $w$ over $R$ (respectively $\Gamma$) is a diagram $D$ for $w$ over $R$ (respectively $\Gamma$) such that
\begin{itemize}
\item the area of $D$ is minimal among all diagrams for $w$ over $R$ (respectively $\Gamma$), and
\item the number of edges of $D$ is minimal among all such diagrams of minimal area.
\end{itemize}

The following theorem is a version of the so-called ``van Kampen Lemma" \cite[Section V.1]{LS}, which is central in small cancellation theory.

\begin{thm}
Let $G$ be a group given by the presentation $\langle S|R\rangle$. An element $w$ of the free monoid $M(S)$ over $S$ satisfies $w=1$ in $G$ if and only if there exists a singular disk diagram over $R$ such that the boundary label of $D$ is $w$.  
\end{thm}

An \emph{arc} in a diagram is an embedded path all interior vertices of which have degree 2 and whose initial and terminal vertices have degrees different from 2. A \emph{spur} is an arc which has a vertex of degree $1$. A face $f$ of a diagram $D$ is called \emph{interior} if the intersection $\partial f\cap \partial D$ contains no edge. All other faces are called \emph{exterior} or \emph{boundary faces}.

Let $p,q$ be positive integers. A diagram $D$ is a \emph{$(p,q)$-diagram} if every interior vertex has degree at least $p$ and if the boundary of every interior face consists of at least $q$ edges. If $D$ is a diagram in which the boundary path of each interior face is the concatenation of at least $q$ arcs, we can consider $D$ as a $(3,q)$-diagram if we ``forget" vertices of degree two. Forgetting a vertex $v$ of degree two means removing the vertex and replacing the two edges incident at $v$ by a single edge (cf. Figure \ref{figure:forget}). In this context, labels and orientations of edges will play no role and will therefore be ignored.

A diagram is a \emph{$[p,q]$-diagram} if every interior vertex has degree at least $p$ and if the boundary of every face consists of at least $q$ edges. Note that a spherical $(p,q)$-diagram is a $[p,q]$-diagram.

\begin{figure}[ht]
\begin{tikzpicture}[>=stealth,shorten <=2.5pt, shorten >=2.5pt]
\coordinate (A1) at (0,1);
\coordinate (A2) at (0,-1);
\coordinate (A3) at (1,0);
\coordinate (A4) at (2,0);
\coordinate (A5) at (3,0);
\coordinate (A6) at (4,1);
\coordinate (A7) at (4,-1);

\fill (A3) circle (2pt);
\fill (A4) circle (2pt);
\fill (A5) circle (2pt);

\draw[-] (A1) to (A3);
\draw[-] (A2) to (A3);
\draw[-] (A3) to (A4);
\draw[-] (A4) to (A5);
\draw[-] (A5) to (A6);
\draw[-] (A5) to (A7);

\coordinate (B1) at ($(A1)+(6,0)$);
\coordinate (B2) at ($(A2)+(6,0)$);
\coordinate (B3) at ($(A3)+(6,0)$);
\coordinate (B4) at ($(A4)+(6,0)$);
\coordinate (B5) at ($(A5)+(6,0)$);
\coordinate (B6) at ($(A6)+(6,0)$);
\coordinate (B7) at ($(A7)+(6,0)$);

\fill (B3) circle (2pt);
\fill (B5) circle (2pt);

\draw[-] (B1) to (B3);
\draw[-] (B2) to (B3);
\draw[-] (B3) to (B5);
\draw[-] (B5) to (B6);
\draw[-] (B5) to (B7);
\end{tikzpicture}
\caption{Forgetting a vertex of degree two.}
\label{figure:forget}
\end{figure}
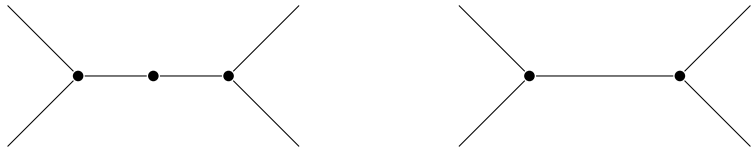

The following are standard results from classical small cancellation theory which we will use later on:

\begin{thm}[Curvature formula I, cf. {\cite[Corollary V.3.4]{LS}}] 
Let $D$ be a $(3,6)$-singular disk diagram with at least two faces. For a face $f$, let $i(f)$ denote the number of its interior edges. Then
 $$\sum_{f\mathrm{\ a\ boundary\ face\ of\ } D} (4-i(f))\geq 6.$$
\end{thm}

In fact, more refined versions of this result can be found in \cite{MW} and \cite{Str}.

\begin{thm}[Curvature formula II, cf. {\cite[Corollary V.3.3]{LS}}] Let $D$ be a $[3,6]$-singular disk diagram. For a vertex $v$, let $d(v)$ denote its degree. Then
 $$\sum_{v\in\partial D}(2+\frac{1}{2}-d(v))\geq 3,$$
where the sum is taken over all vertices in the boundary.
\end{thm}

The following theorem is an immediate conclusion of \cite[Theorem 6.2]{LS} and the first curvature formula.
\begin{thm}[Area theorem I]
Let $D$ be a $(3,6)$-singular disk diagram. Then
$$ |D|\leq 3|\partial D|^2,$$
where $|D|$ is the number of faces of $D$ and $|\partial D|$ is the number of its boundary edges.
\end{thm}

A stronger statement holds for $(3,7)$-diagrams. It is proved in \cite[Proposition 2.7]{Str}.
\begin{thm}[Area theorem II]
 Let $D$ be a $(3,7)$-singular disk diagram. Then
$$ |D|\leq 8 |\partial D|.$$
\end{thm}

\subsection{Basic tools and results}\label{subsection:p3}

We will now explain how to apply the aforementioned results from classical small cancellation theory to graphical small cancellation presentations. The following definition is given in \cite{Oll}.

\begin{defi}[To originate from $\Gamma$] Let $\Gamma$ be a $C(2)$-labelled graph, and let $f$ be a face in a diagram $D$ over $\Gamma$. Choose a basepoint and an orientation for $\partial f$. Then $\partial f$ maps (we say \emph{lifts}) to a cycle in $\Gamma$, and since $C(2)$ is satisfied, this cycle is unique. Let $p$ be a path in an interior arc of $D$ that lies in the intersection of faces $f_1$ and $f_2$. Then we can choose basepoints and orientations for $\partial f_1$ and $\partial f_2$, such that $p$ is an initial subpath of $\partial f_1$ and $\partial f_2$. If the images of $p$ via the lifts of $\partial f_1$ and $\partial f_2$ to $\Gamma$ coincide, we say $p$ \emph{originates from} $\Gamma$.
\end{defi}

Observe that any interior arc of a diagram that does not originate from $\Gamma$ is a piece.

The above definition can be extended to $Gr(2)$-labelled graphs. In that case, lifts are considered up to composition with an automorphism of $\Gamma$. Any interior arc of a diagram that does not essentially originate from $\Gamma$ is an essential piece. As stated before, all proofs and conclusions of Subsections \ref{subsection:p1}-\ref{subsection:p3} for $C(n)$- (respectively $C'(n)$-) labelled graphs are true for $Gr(n)$- (respectively $Gr'(\lambda)$-) labelled graphs as well. We will not state them doubly to avoid confusing notation.

\vspace{8pt}

The following lemma is our main tool in extending classical small cancellation theory to graphical small cancellation theory. It is based on \cite{Oll}, in which the $C'(\frac{1}{6})$ condition is investigated.

\begin{lem}\label{lem:graphical_basic} Let $D$ be a singular disk diagram over a $C(n)$-labelled graph $\Gamma$, where $n\geq 6$. Then one of the following holds:
\begin{itemize}
 \item Removing all edges of $D$ originating from $\Gamma$ and merging the corresponding 2-cells yields a $(3,n)$-diagram $D'$ over $\Gamma$, all faces of which are simply connected.
 \item $D$ has a subdiagram that is a simple disk diagram with at least one face, with freely trivial boundary word, all interior edges of which originate from $\Gamma$ and no boundary edge of which is an interior edge of $D$ that originates from $\Gamma$.
\end{itemize}
\end{lem}

\begin{proof} Note that removing an arc $a$ from a diagram $D$ may yield a resulting space $D'$ that is not a diagram. This can be the case if $a$ is the image of two distinct, non-consecutive subpaths of $\partial f$ for a single face $f$. In this case, after removing the arc $a$, $f$ becomes the image of a closed annulus. If multiple arcs are removed, this can happen multiple times. Then $f$ becomes the image of a closed disk with holes, i.e. the image of a space $b'=\overline b\setminus \sqcup_{i=1}^nb_i$, where $b=D^2\subset \R^2$, $n\geq 1$ and each $b_i$ is an open disk of radius $\epsilon$ contained in $b$ such that $\overline b_i\cap \partial b=\emptyset$ and for $i\neq j: \overline b_i\cap \overline b_j=\emptyset$.  We call such a space $f$ a \emph{face with holes}.

Let $D$ be as in the statement and assume the second claim does not hold. Let $D'$ be obtained from $D$ by deleting all edges originating from $\Gamma$ and merging the corresponding faces. We will show that there are no faces with holes, and hence $D'$ is a diagram. Moreover, there may be faces in $D'$ that are not simply connected. We rule out this case as well. (See Figure \ref{figure:holey} for illustrations.)

There are no spurs in $D'$ whose endpoints lie in the interior of $D'$, as such a spur would lie entirely in one face and, by the reducedness of the labelling, consist of edges originating from $\Gamma$. Assume there is a face with holes or a non-simply connected face $f$. By choosing $f$ to be innermost, we can assume that there is a subdiagram $\Delta$ which is enclosed by $f$ (i.e. the interior of $\Delta$ lies in the union of bounded connected components of $\R^2\setminus f$), such that $\Delta$ is in fact a diagram, $\Delta$ has at least one face, all faces of $\Delta$ are simply connected and $\Delta$ has no spur.

Since $\Delta$ is enclosed by $f$, $\Delta$ has at most one vertex (say $v$) that meets an edge that is not in $\Delta$. By assumption, all faces of $\Delta$ bear the labels of nontrivial cycles in $\Gamma$. Note that all interior arcs of $D'$ are pieces. Now we can forget vertices of degree two in $D'$. This turns the subdiagram $\Delta$ into a $[3,6]$-diagram, due to our small cancellation assumption. There is at most one vertex in $\Delta$ (namely $v$) that has degree (in $\Delta$) less than $3$. This contradicts the second curvature formula for $\Delta$.

Hence $D'$ is a diagram and all faces are simply connected. Since the second claim does not hold, every face bears a freely nontrivial boundary word, i.e. a word read on a nontrivial cycle of $\Gamma$. The small cancellation assumption and the observation that interior arcs are pieces yield that $D'$ is a $(3,n)$-diagram.
\end{proof}

\begin{figure}
\begin{tikzpicture}
 \draw (0,0) circle (1);
 \draw (0,0) circle (2);
 \draw[dashed,shorten <=1pt, shorten >=1pt] (1,0) -- (2,0);

 \draw (6.1,0) circle (1);
 \draw (6.1,0) circle (2);
 \draw (7.1,0) -- (8.1,0);

 \node at (0,0) {\small $\Delta$};
 \node at (-1.5,0) {\small $f$};

 \node at (6.1,0) {\small $\Delta$};
 \node at (4.6,0) {\small $f$};

\end{tikzpicture}
\caption{On the left: A face $f$ with one hole enclosing the subdiagram $\Delta$. The dashed line represents an arc in the original diagram $D$ that was removed when passing to $D'$. On the right: A non-simply connected face $f$ enclosing the subdiagram $\Delta$.}
\label{figure:holey}
\end{figure}
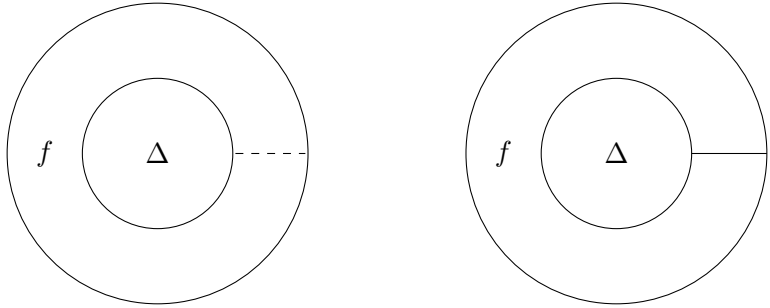

Suppose in a singular disk diagram $D$, there is a simply connected face $f$ that has freely trivial boundary word and such that there are no spurs whose endpoints lie inside $f$. Then we can fold the boundary of $f$, thus removing $f$ without altering anything else in the diagram. This is done as follows:

Since the boundary cycle $\partial f$ embeds into $D$, any two consecutive edges with inverse labels in $\partial f$ can folded together (i.e. identified) in $D$ such that the resulting edge only intersects the resulting face $f'$ in a vertex. If $f$ had only two boundary edges, we let $f'$ be a vertex. The recursive application of this folding turns $f$ into a labelled tree, which, considered as a diagram, has the same boundary word as $f$. See Figure \ref{figure:fold} for an example.

\begin{figure}[ht]
\begin{tikzpicture}[>=stealth,shorten <=2.5pt, shorten >=2.5pt]
\coordinate [label=above:$f$] (A) at (0,-.3);
\node[coordinate] at (0  :1.5) (A0) {};
\node[coordinate] at (45 :1.5) (A1) {};
\node[coordinate] at (90 :1.5) (A2) {};
\node[coordinate] at (135:1.5) (A3) {};
\node[coordinate] at (180:1.5) (A4) {};
\node[coordinate] at (225:1.5) (A5) {};
\node[coordinate] at (270:1.5) (A6) {};
\node[coordinate] at (315:1.5) (A7) {};
\fill (A1) circle (2pt);\fill (A2) circle (2pt);\fill (A3) circle (2pt);
\fill (A4) circle (2pt);\fill (A5) circle (2pt);\fill (A6) circle (2pt);
\fill (A7) circle (2pt);\fill (A0) circle (2pt);

\draw[->] (A1) to node [auto] {\small $a$} (A0);
\draw[->] (A2) to node [auto] {\small $b$} (A1);
\draw[->] (A2) to node [auto,swap] {\small $c$} (A3);
\draw[->] (A4) to node [auto] {\small $c$} (A3);
\draw[->] (A5) to node [auto] {\small $d$} (A4);
\draw[->] (A5) to node [auto,swap] {\small $d$} (A6);
\draw[->] (A6) to node [auto,swap] {\small $b$} (A7);
\draw[->] (A7) to node [auto,swap] {\small $a$} (A0);

\node[coordinate] at (0  :6) (N) {};
\node[coordinate] at ($(N)+(0  :1.14805)$) (B01) {};
\node[coordinate] at ($(N)+(0  :2.2961 )$) (B02) {};
\node[coordinate] at ($(N)+(135:1.14805)$) (B3) {};
\node[coordinate] at ($(N)+(225:1.14805)$) (B5) {};

\fill (B3) circle (2pt);
\fill (B5) circle (2pt);
\fill (B01) circle (2pt);
\fill (B02) circle (2pt);
\fill (N) circle (2pt);

\draw[->] (N) to node [auto] {\small $b$} (B01);
\draw[->] (B01) to node [auto] {\small $a$} (B02);
\draw[->] (N) to node [auto,swap] {\small $c$} (B3);
\draw[->] (B5) to node [auto,swap] {\small $d$} (N);
\end{tikzpicture}
\caption{Folding a face $f$ with freely trivial boundary word.}
\label{figure:fold}
\end{figure}
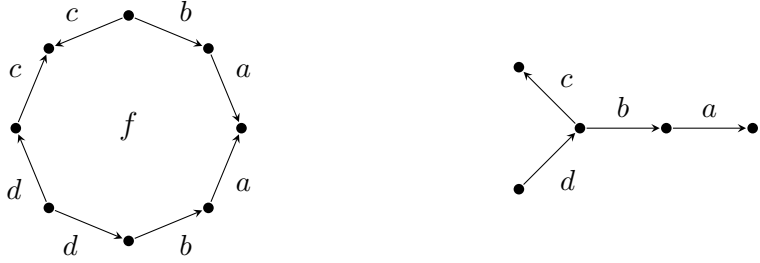

After folding away $f$, the resulting diagram has smaller area than the original one. This observation also applies to any subdiagram $\Delta$ of $D$ with at least one face and with freely trivial boundary word, since we can simply replace $\Delta$ by a face $f$ with the same boundary word and then fold away $f$. Hence removing all edges originating from $\Gamma$ on a diagram $D$ of minimal area for a word $w$ yields a diagram as in the first case of Lemma \ref{lem:graphical_basic}. We therefore obtain:

\begin{cor}\label{cor:graphical_basic} Let $w\in M(S)$ satisfying $w=1$ in $G(\Gamma)$, and let $D$ be a minimal diagram for $w$ over $\Gamma$. Then no interior edge of $D$ originates from $\Gamma$, and therefore, $D$ is a $(3,n)$-diagram, all faces of which are simply connected.
\end{cor}

The following lemma is proved just as Lemma \ref{lem:graphical_basic} with the additional observation that there is no spherical $[3,6]$-diagram with more than one face. This follows from the second curvature formula.

\begin{lem}\label{lem:graphical_basic_spherical} Let $D$ be a simple spherical diagram over a $C(6)$-labelled graph $\Gamma$. Then one of the following holds:
\begin{itemize}
 \item All edges of $D$ originate from $\Gamma$.
 \item $D$ has a subdiagram that is a simple disk diagram with at least one face, with freely trivial boundary word, all interior edges of which originate from $\Gamma$ and no boundary edge of which originates from $\Gamma$.
\end{itemize}
\end{lem}

The set of all words read on nontrivial cycles of a labelled graph $\Gamma$ is generally infinite. The set of all words read on all \emph{simple} cycles of $\Gamma$ generates the same normal subgroup of $F(S)$. If $\Gamma$ has finitely generated fundamental group, this set is finite. The following lemma improves the conclusions of Lemma \ref{lem:graphical_basic} to show, that considering the presentation given by the words read on all simple cycles is sufficient to obtain $(3,n)$-diagrams from $C(n)$ labelled graphs if $n\geq 6$.

\begin{lem}\label{lem:graphical_simple} Let $\Gamma$ be a $C(n)$ labelled graph for $n\geq 6$. Let $w\in M(S)$ satisfying $w=1$ in $G(\Gamma)$. Then there exists a diagram $D$ for $w$ such that no interior edge of $D$ originates from $\Gamma$ and such that every face bears the label of a simple cycle in $\Gamma$.
\end{lem}
Note that by Lemma \ref{lem:graphical_basic}, $D$ is a $(3,n)$ diagram, all faces of which are simply connected.

\begin{proof}[Proof of the lemma] Let $D$ be a diagram for $w$ over $\Gamma$ with a minimal number of edges whose number of vertices is minimal among all diagrams for $w$ over $\Gamma$ with a minimal number of edges. 

Suppose $D$ has a subdiagram $\Delta$ as in the second case of Lemma \ref{lem:graphical_basic}. Then we can remove the interior edges of $\Delta$ and fold its boundary thus removing $\Delta$, leaving a diagram over $\Gamma$. This diagram has fewer edges than $D$ while having the same boundary word, which contradicts the assumptions.

Thus we are in the first case of Lemma \ref{lem:graphical_basic}. Since deleting all interior edges originating from $\Gamma$ gives a diagram over $\Gamma$, the minimality assumptions yield that no interior edge originates from $\Gamma$.  Hence there is no spur whose endpoint lies in the interior of $D$, and, by Lemma \ref{lem:graphical_basic}, every face is simply connected.

Suppose $D$ has a face $f$ bearing a boundary word that cannot be read on a simple cycle. Since $f$ is simply connected and since there are no spurs inside $f$, $\partial f$ embeds into $D$. Thus there are two distinct vertices $v_1$ and $v_2$ in $\partial f$ mapping to the same vertex in $\Gamma$ via the lift of $\partial f$ to $\Gamma$. We can identify $v_1$ and $v_2$, i.e.\ ``pinch" $v_1$ and $v_2$ together to a single vertex (cf. Figure \ref{figure:pinching}). 

If faces with freely trivial boundary words arise, they can be folded and removed to leave a diagram over $\Gamma$. This contradicts the minimality of number of edges of $D$. If no face with freely trivial boundary arises, the diagram obtained by pinching vertices is a diagram over $\Gamma$ with the same number of edges as $D$, but with fewer vertices, again contradicting the assumptions.
\end{proof}

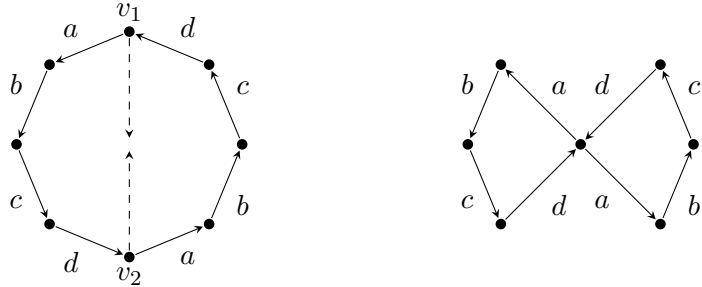
\begin{figure}[ht]
\begin{tikzpicture}[>=stealth,shorten <=2.5pt, shorten >=2.5pt]
\coordinate (A0) at (0  :1.5) {};
\coordinate (A1) at (45 :1.5) {};
\coordinate [label=above:\small $v_1$] (A2) at (90 :1.5) {};
\coordinate (A3) at (135:1.5) {};
\coordinate (A4) at (180:1.5) {};
\coordinate (A5) at (225:1.5) {};
\coordinate [label=below:\small $v_2$] (A6) at (270:1.5) {};
\coordinate (A7) at (315:1.5) {};

\fill (A1) circle (2pt);\fill (A2) circle (2pt);\fill (A3) circle (2pt);
\fill (A4) circle (2pt);\fill (A5) circle (2pt);\fill (A6) circle (2pt);
\fill (A7) circle (2pt);\fill (A0) circle (2pt);

\draw[->] (A0) to node [auto,swap] {\small $c$} (A1);
\draw[->] (A1) to node [auto,swap] {\small $d$} (A2);
\draw[->] (A2) to node [auto,swap] {\small $a$} (A3);
\draw[->] (A3) to node [auto,swap] {\small $b$} (A4);
\draw[->] (A4) to node [auto,swap] {\small $c$} (A5);
\draw[->] (A5) to node [auto,swap] {\small $d$} (A6);
\draw[->] (A6) to node [auto,swap] {\small $a$} (A7);
\draw[->] (A7) to node [auto,swap] {\small $b$} (A0);
\draw[->,dashed] (A2) to (0,0);
\draw[->,dashed] (A6) to (0,0);

\coordinate (N) at (0:6) {};
\coordinate (B0) at ($(N)+(A0)$) {};
\coordinate (B1) at ($(N)+(A1)$) {};
\coordinate (B3) at ($(N)+(A3)$) {};
\coordinate (B4) at ($(N)+(A4)$) {};
\coordinate (B5) at ($(N)+(A5)$) {};
\coordinate (B7) at ($(N)+(A7)$) {};

\fill (B0) circle (2pt);
\fill (B1) circle (2pt);
\fill (B3) circle (2pt);
\fill (B4) circle (2pt);
\fill (B5) circle (2pt);
\fill (B7) circle (2pt);
\fill (N) circle (2pt);

\draw[->] (N) to node [auto,swap] {\small $a$} (B7);
\draw[->] (B7) to node [auto,swap] {\small $b$} (B0);
\draw[->] (B0) to node [auto,swap] {\small $c$} (B1);
\draw[->] (B1) to node [auto,swap] {\small $d$} (N);
\draw[->] (N) to node [auto,swap] {\small $a$} (B3);
\draw[->] (B3) to node [auto,swap] {\small $b$} (B4);
\draw[->] (B4) to node [auto,swap] {\small $c$} (B5);
\draw[->] (B5) to node [auto,swap] {\small $d$} (N);

\end{tikzpicture}
\caption{Pinching together two vertices lying on a face that has a boundary path not lifting to a simple cycle.}
\label{figure:pinching}
\end{figure}

\begin{defi}[Gromov-hyperbolic group]\label{defi:hyperbolic_group}
Let $X$ be a geodesic metric space. A \emph{geodesic triangle} in $X$ is a triple of geodesics $(g_1,g_2,g_3)$ such that the endpoint $\tau(g_1)$ of $g_1$ equals the starting point $\iota(g_2)$ of $g_2$, $\tau(g_2)=\iota(g_3)$ and $\tau(g_3)=\iota(g_1)$. Let $\delta\geq 0$. Then $X$ is called \emph{$\delta$-hyperbolic} if and only if for all geodesic triangles $(g_1,g_2,g_3)$, $g_3$ is contained in the $\delta$-neighborhood of $g_1\cup g_2$. The space $X$ is called \emph{hyperbolic} if it is $\delta$-hyperbolic for some $\delta\geq0$.

Let $G$ be a group generated by a finite set $S$. The Cayley graph $\Cay(G,S)$ of $G$ with respect to $S$ is a connected graph endowed with the graph metric on its vertex set. If we consider each edge of $\Cay(G,S)$ isometric to the unit interval $[0,1]$, $\Cay(G,S)$ becomes a geodesic metric space. We say $G$ is $\delta$-hyperbolic with respect to $S$ if $\Cay(G,S)$ is $\delta$-hyperbolic. The group $G$ is called \emph{Gromov-hyperbolic} if it is $\delta$-hyperbolic with respect to $S$ for some $\delta\geq 0$ and some finite generating set $S$.
\end{defi}

\begin{defi}[Isoperimetric inequality] Let $G$ be a group given by the presentation $\langle S|R\rangle$. For $w\in M(S)$ satisfying $w=1$ in $G$, let $\operatorname{Area}_R(w)$ denote the minimal area of a singular disk diagram for $w$ over $R$. 
The \emph{Dehn function} of $G$ with respect to $\langle S|R\rangle$ is the map $f:\N\to \N$ given by 
$$f(l)=\max\{\operatorname{Area}_R(w)| |w|=l\}.$$
We say a group presentation satisfies a \emph{linear} (resp. \emph{quadratic}) \emph{isoperimetric inequality} if the corresponding Dehn function is bounded from above by a linear (resp. quadratic) map $\R\to\R$.
\end{defi}

We state two immediate consequences of Lemma \ref{lem:graphical_simple} using the area theorems. For classical $C(6)$ and $C(7)$ small cancellation presentations, these are well-known. It is shown in \cite[Theorem 2.5 and Proposition 2.10]{Al} that a group is Gromov-hyperbolic if and only if it has a finite presentation satisfying a linear isoperimetric inequality. 

We choose to state the theorems for the $Gr(7)$ and $Gr(6)$ conditions (instead of the graphical $C(7)$ and $C(6)$ conditions) as they are more general.

\begin{thm}\label{thm:linear_isoperimetry}
 Let $\Gamma$ be a $Gr(7)$-labelled graph. Let $R$ be the set of words read on all simple cycles of $\Gamma$. Then the presentation $\langle S|R\rangle$ satisfies the linear isoperimetric inequality:
$$\operatorname{Area}_R(w)\leq 8|w|.$$
In particular, if $\Gamma$ is finite, then the group $G(\Gamma)$ is Gromov-hyperbolic.
\end{thm}

\begin{thm}\label{thm:quadratic_isoperimetry}
 Let $\Gamma$ be a $Gr(6)$-labelled graph. Let $R$ be the set of words read on all simple cycles of $\Gamma$. Then the presentation $\langle S|R\rangle$ satisfies the quadratic isoperimetric inequality:
$$\operatorname{Area}_R(w)\leq 3|w|^2.$$
\end{thm}

%% file: 3asphericity.tex
\subsection{Asphericity}

For a group presentation $\langle S|R \rangle$, the \emph{presentation complex} is a CW-complex constructed as follows: The 0-skeleton is a single point $v$. For each generator in $S$, a labelled, oriented loop is glued to $v$, and for each relator $r$ in $R$, a 2-cell whose boundary label is $r$ is glued along its boundary label onto the 1-skeleton. A CW-complex is called \emph{aspherical} if its universal cover is contractible.

We show that a group given by a graphical $C(6)$ presentation has an aspherical presentation complex. This is a generalization of the corresponding well-known result for classical $C(6)$ presentations where no relator is a proper power. (For a proof, see \cite[Theorem 13.3]{Ols} and \cite{CCH}.) It can easily be seen from Example \ref{ex:cayley} that this result does not hold for $Gr(6)$ presentations in general. 

\begin{thm}\label{thm:asphericity} Let $\Gamma$ be a $C(6)$-labelled graph. Let $R$ be the set of cyclic reductions of words read on free generating sets of the fundamental groups of the connected components of $\Gamma$. Then the presentation complex associated to $\langle S|R\rangle$ is aspherical.
\end{thm}
Together with \cite[§VIII.2]{Br} this yields:
\begin{cor} The group $G(\Gamma)$ defined by $\Gamma$ has an at most 2-dimensional $K(G(\Gamma),1)$ space and hence cohomological dimension at most 2. Therefore, $G(\Gamma)$ is torsion-free. 
\end{cor}

\begin{lem}\label{lem:infinite_presentation} Let $\Gamma$ and $R$ be as in Theorem \ref{thm:asphericity}. Let $R'$ be a proper subset of $R$. Then the homomorphism $\langle S|R'\rangle\to\langle S|R\rangle$ induced by the identity on $S$ is not injective. In particular, if $R$ is infinite, then $G(\Gamma)$ is not finitely presented.
\end{lem}
We postpone the proof of this lemma as it will be derived from our proof of Theorem \ref{thm:asphericity}. To prove asphericity, we will show that spherical diagrams over a $C(6)$ graph are reducible in an appropriate sense. We use definitions and results from \cite{CH} that allow an algebraic treatment of diagrams:

Let $R\subset F(S)$, and set $R^S:=\{gr^\epsilon g^{-1}|r\in R, g\in F(S),\epsilon\in\{\pm1\}\}\subset F(S)$. A \emph{sequence over $R$} is a finite sequence of elements of $R^S$. On all sequences over $R$, we consider the following operations:
\begin{itemize}
 \item Exchange: Replace a pair $(x,y)$ by $(xyx^{-1},x)$ or by $(y,y^{-1}xy)$.
 \item Deletion: Delete a pair $(x,x^{-1})$.
 \item Insertion: Insert at any position a pair $(x,x^{-1})$ for any $x\in R^S$.
\end{itemize}
We call two sequences over $R$ \emph{equivalent} if one can be transformed into the other by a finite sequence of these operations. We call any finite sequence over $R$ an \emph{identity sequence} if the product of its elements (taken in the order they appear in the sequence) is trivial in $F(S)$. We call a sequence \emph{trivial} if it is equivalent to the empty sequence. 

To a diagram $(D,v)$ over $R$, we can associate a sequence $\Sigma$ over $R$ as follows: From $D$, we construct a singular disk diagram $D'$ by ``ungluing'' faces of $D$ along edges, such that $D'$ is a bouquet of faces, each connected to the basepoint $v$ by an arc, such that the boundary word of $D'$ is freely equal to that of $D$. If $D$ is spherical, we choose some embedding in the plane after ungluing the first edge. The boundary word of $D'$ is freely trivial in that case.

The sequence $\Sigma$ is then the sequence of the labels of faces of $D'$, each conjugated by the label of the path that connects it to $v$, read in the order in which they appear in the boundary path. The resulting sequence is called a \emph{derived sequence} for $(D,v)$.  Its length is equal to the number of faces of $D$, and the product over all elements of $\Sigma$ is freely equal to the boundary word of $D$. 

Whilst a derived sequence for a diagram $(D,v)$ is not unique, \cite[Proposition 8]{CH} yields that any two derived sequences for a singular disk diagram $(D,v)$ are equivalent. Moreover, \cite[Corollary 1 of Proposition 8]{CH} implies that if $\Sigma$ and $\Sigma'$ are derived sequences for a spherical diagram $D$, $\Sigma$ is trivial if and only if $\Sigma'$ is.

For any sequence $\Sigma$ over $R$, we can construct a singular disk diagram $D$ over $R$ which has $\Sigma$ as a derived sequence. (This is done by reversing the procedure described above.) If $\Sigma$ is an identity sequence, the diagram has freely trivial boundary word. This freely trivial boundary word can be ``sewn up'' to obtain a spherical diagram. Sewing up means that, whenever there are two edges bearing inverse labels in the boundary, we fold them together until eventually we are left with a spherical diagram. (See \cite[Section 1.5]{CH} for details.) We call a diagram constructed from a sequence $\Sigma$ an \emph{associated diagram} for $\Sigma$

The following theorem is deduced from \cite[Proposition 1.3 and Proposition 1.5]{CCH}. A presentation $\langle S|R\rangle$ is \emph{concise} if for any relator $r\in R$, no $r'\in R$ that is distinct from $r$ is conjugate to $r$ or $r^{-1}$.

\begin{thm} Let $\langle S|R\rangle$ be a presentation where each relator is nontrivial and freely reduced. Then the associated presentation complex is aspherical if and only if:
 \begin{itemize}
  \item The presentation is concise,
  \item no relator is a proper power,
  \item and any identity sequence over $R$ is trivial.
 \end{itemize}
\end{thm}

We will show that the presentation in our theorem satisfies these three conditions.

\begin{lem}\label{lem:concise} Let $\Gamma$ be a $C(2)$-labelled graph. Let $R$ be the set of cyclic reductions of words read on free generating sets of the fundamental groups of the connected components of $\Gamma$ (with respect to some base points). 
Then $R$ is concise and contains no proper powers. 
\end{lem}

\begin{proof}
For any cycle $\gamma$ on $\Gamma$, there exists a maximal initial subpath $p$ such that $\gamma=p\gamma'p^{-1}$ for some subpath $\gamma'$. Then $\gamma'$ is a cycle and, since the labelling of $\Gamma$ is reduced, its label $\ell(\gamma')$ equals the cyclic reduction of $\ell(\gamma)$. We call $\gamma'$ the \emph{cyclic reduction} of $\gamma$.

Let $\Gamma_i$ for $i\in I$, where $I$ is some index set, denote the connected components of $\Gamma$, and let $\nu_i$ denote a base vertex in $\Gamma_i$ for each $i$. Let $\gamma$ be the reduced representative of an element of a free generating set of some $\pi_1(\Gamma_i,\nu_i)$ and $\gamma'$ its cyclic reduction such that $\gamma=p\gamma' p^{-1}$ for some path $p$. Suppose $\ell(\gamma')=w^n$ for $w\neq 1$ and $n>1$. Let $\gamma'=\gamma_0...\gamma_{n-1}$ such that $\ell(\gamma_j)=w$ for $j\in \Z/n\Z$. 

Then the map of labelled graphs $\gamma'\to\gamma'$ induced by $\gamma_j\to\gamma_{j+1}$ must be the identity, for otherwise $\gamma'$ would have two distinct maps of labelled graphs to $\Gamma$ and thus would be a single piece. Therefore, $\gamma_j=\gamma_{j+1}$ for all $j$, and each $\gamma_j$ is a cycle. Thus $[\gamma]=[p\gamma_0 p^{-1}]^n$ in $\pi_1(\Gamma_i,\nu_i)$. But it is easy to see that an element of a free generating set of $\pi_1(\Gamma_i,\nu_i)$ cannot be a proper power.

Now assume that there are two distinct relators $r$ and $r'$ in $R$ coming from cycles $[\gamma]$ and $[\delta]$ in the free generating sets, where $\gamma$ and $\delta$ are reduced, such that $r'$ is conjugate to $r$ (or $r^{-1}$). Since $r$ and $r'$ are cyclically reduced, they coincide up to a cyclic permutation (and possibly inversion). Again, consider the cyclic reductions $\gamma'$ and $\delta'$. Then we can perform a cyclic shift on $\delta'$ such that the resulting cycle $\tilde \delta '$ has the same label as $\gamma'$ (or $(\gamma')^{-1}$). But then $\tilde \delta'$ must equal $\gamma'$ (or $(\gamma')^{-1})$) for otherwise, it would be a single piece. This implies that $[\gamma]$ and $[\delta]$ lie on the same connected component of $\Gamma$, say $\Gamma_i$, and that they are conjugate in $\pi_1(\Gamma_i,\nu_i)$ (or conjugate up to inversion). This cannot hold for two elements of a free generating set.
\end{proof}

\begin{lem}\label{lem:asphericity}
Let $\Gamma$ be a connected $C(2)$-labelled graph. Let $R$ be the set of cyclic reductions of words read a set of free generators of $\pi_1(\Gamma,\nu)$ for some $\nu\in\Gamma$. Let $(D,v)$ be a simple disk diagram over $R$ with freely trivial boundary word such that every interior edge originates from $\Gamma$. Then any derived sequence is a trivial identity sequence.
\end{lem}
\begin{proof} Let $\{\phi_i|i\in I\}$ be a set of free generators of $\pi_1(\Gamma,\nu)$ such that
$$\phi_i=[\gamma_i \rho_i\gamma_i^{-1}],$$
where $\ell(\rho_i)$ is equal to the cyclic reduction of $\ell_*(\phi_i)$ and $\gamma_i$ is a reduced path. Denote the terminal vertex of $\gamma_i$ by $\nu_i$. Then $\{\ell(\rho_i)|i\in I\}=R$, and $\{\ell_*(\phi_i)|i\in I\}\subset R^S$.

The diagram $D$ globally lifts to $\Gamma$ via a map $\lambda:D^{1}\to \Gamma$, where $D^1$ denotes the 1-skeleton of $D$. Let $v\in\partial D$. Since $\Gamma$ is connected, we may assume that $\nu$ was chosen such that $\nu=\lambda(v)$. The path $\lambda(\partial D)$ is a cycle on $\Gamma$ based at $\nu$. Since the labelling of $\Gamma$ is reduced, the assumption that the boundary label of $D$ is freely trivial implies that $\lambda(D)$ is 0-homotopic, i.e.\ $[\lambda(\partial D)]=1\in\pi_1(\Gamma,\nu)$.

We number the faces of $D$ as $f_1,...,f_n$ and choose an orientation on each face. For a face $f_j$ with boundary label $\ell(\rho_i)^{\epsilon_j}$, where $i\in I$ and $\epsilon_j\in\{\pm1\}$, let $v_j$ denote the vertex that lifts to $\nu_i$, and set $t_j:=i$. Then there are paths $p_j$ in $D^1$ from $v$ to $v_j$ such that in $\pi_1(\Gamma,\nu)$:
\begin{equation*}
 1=[\lambda(\partial D)]=[\lambda(p_1)\rho_{t_1}^{\epsilon_1}\lambda(p_1)^{-1}]\ldots [\lambda(p_n)\rho_{t_n}^{\epsilon_n}\lambda(p_n)^{-1}]. 
\end{equation*}
For each $j$, the path $\lambda(p_j)\gamma_{t_j}^{-1}$ is a cycle based at $\nu$; we denote $\eta_j:=[\lambda(p_j)\gamma_{t_j}^{-1}]$. The sequence
$$\Sigma:=(\ell_*(\eta_1\phi_{t_1}^{\epsilon_1}\eta_1^{-1}),\ldots)=(\ell_*(\eta_1)\ell_*(\phi_{t_1}^{\epsilon_1})\ell_*(\eta_1^{-1}),\ldots)$$
is a derived sequence for $(D,v)$. For each $j$, we can express $\eta_j$ in the free generators $\{\phi_i|i\in I\}$ as a reduced word $W_j$ and write:
$$\Sigma=(\ell_*(W_1)\ell_*(\phi_{t_1}^{\epsilon_1})\ell_*(W_1^{-1}),\ldots ,\ell_*(W_n)\ell_*(\phi_{t_n}^{\epsilon_n})\ell_*(W_n^{-1})).$$

Now suppose $W_1\neq1$, and let the first letter of $W_1$ be $\phi_{f_1}^{e_1}$. Then we can insert the pair $(\ell_*(\phi_{f_1}^{e_1}),\ell_*(\phi_{f_1}^{-e_1}))$ into $\Sigma$ at the first position and perform an exchange operation to obtain
$$(\ell_*(\phi_{f_1}^{e_1}),\ell_*(W_1')\ell_*(\phi_{t_1}^{\epsilon_1})\ell_*(W_1'^{-1}),\ell_*(\phi_{f_1}^{-e_1}),\ell_*(W_2)\ell_*(\phi_{t_2}^{\epsilon_2})\ell_*(W_2^{-1}),\ldots ),$$
where $\phi_{f_1}^{e_1}W_1'=W_1$, and $W_1'$ is shorter than $W_1$. Iterating this procedure and applying it to all $W_i$ yields a sequence $\Sigma'$ of the form
$$\Sigma'=(\ell_*(\phi_{g_1}^{h_1}),\ldots ,\ell_*(\phi_{g_N}^{h_N})),$$
where each $h_i\in\{\pm 1\}$. Since $1=\phi_{g_1}^{h_1}\ldots \phi_{g_N}^{h_N}$ in $\pi^1(\Gamma,\nu)$ and since a free reduction on the right hand side of this equation corresponds to a deletion operation in $\Sigma'$, we see that $\Sigma'$ can be transformed into the empty sequence by deletion operations. 
\end{proof}

\begin{lem}
Let $\Gamma$ be a $C(6)$-labelled graph. Let $R$ be the set of cyclic reductions of words read on free generating sets of the fundamental groups of the connected components of $\Gamma$. Then any identity sequence over $R$ is trivial.
\end{lem}

\begin{proof} Let $\Sigma$ be a nontrivial identity sequence over $R$ of minimal length. Then there is an associated spherical diagram $D$. We may restrict ourselves to the case that $D$ is a simple spherical diagram, as the general case can be constructed from this. Lemma \ref{lem:graphical_basic_spherical} implies that all edges of $D$ originate from $\Gamma$, or that $D$ has a subdiagram $f$ that is a simple disk diagram with at least one face and with freely trivial boundary word, and all interior edges of $f$ originate from $\Gamma$. Assume the latter. Note that the 1-skeleton of $f$ maps to a connected component of $\Gamma$.

Changing the base point of a diagram corresponds to a conjugation of all elements of the derived sequence with an element $g\in F(S)$. A sequence is equivalent to a shorter sequence if and only if its conjugate is. Hence we can assume that the base point $v$ lies in the boundary of $f$.

We now associate to $(D,v)$ a derived sequence $\Sigma'$ that has an initial subsequence $\sigma$ that is a derived sequence for $(f,v)$: We cut $f$ out of $D$ and glue $f$ and $D\setminus f$ together at the point $v$ such that the boundary labels of $f$ and $D\setminus f$ read from $v$ are inverse. Let $\Sigma'$ be the concatenation of a sequence $\sigma$ for $f$ and one for $D\setminus f$. Since $f$ has at least one face, $\sigma$ is nonempty. Note that $\Sigma$ and $\Sigma'$ have the same length. By Lemma \ref{lem:asphericity}, $\sigma$ reduces to the trivial sequence, and therefore $\Sigma'$ is equivalent to a shorter sequence. Thus the minimality assumption on $\Sigma$ yields that $\Sigma'$ is trivial. Now, as mentioned before, \cite[Corollary 1 to Proposition 8]{CH} implies that $\Sigma$ is trivial, a contradiction.

If all edges of $D$ originate from $\Gamma$, we can unglue two faces along any edge of $D$ to obtain a simple disk diagram with freely trivial boundary word that globally lifts to $\Gamma$. Lemma \ref{lem:asphericity} yields that any derived sequence is trivial, and therefore $\Sigma$ is trivial.
\end{proof}

\begin{proof}[Proof of Lemma \ref{lem:infinite_presentation}] Let $R$ be as in Theorem \ref{thm:asphericity}. Assume there is $r\in R$ such that $r\in\langle\langle R\setminus \{r\}\rangle\rangle$. Then there exists a diagram for $r$ over $R\setminus \{r\}$ and hence a spherical diagram $D$ over $R$ with exactly one face labelled by $r$. There exists a derived sequence for $D$ containing a conjugate of $r$, but (using Lemma \ref{lem:concise}) no conjugate of $r^{-1}$. Therefore the sequence is not trivial, which is a contradiction.
\end{proof}

%% file: 4free_subgroups.tex
\section{Free subgroups}

We prove that graphical $C(7)$ groups are either cyclic, or they have non-abelian free subgroups. 

\begin{thm}\label{thm:free}
 Let $\Gamma$ be a $C(7)$-labelled graph. Then $G(\Gamma)$ contains a non-abelian free subgroup unless it is trivial or infinite cyclic.
\end{thm}

\begin{cor}
 Graphical $C(7)$ groups have exponential growth and are non-amenable, unless they are trivial or infinite cyclic.
\end{cor}

\begin{remark}\label{rem:obvious} We show that it is easy to check if $G(\Gamma)$ is trivial or infinite cyclic. First, either one of the curvature formulas implies that there is no diagram over $\Gamma$ whose boundary has length 1 and which has more than one face. Hence $G(\Gamma)$ is trivial if and only if it contains a loop of length 1 labelled with $s$ for each $s\in S$.

Second, we check when an infinite cyclic group can arise. From the labelled graph $\Gamma$, we construct a graph $\Gamma'$ by iteratively removing every edge that is contained in a simple cycle and that is not a piece. We also remove the letters labelling those edges from the alphabet $S$ to obtain the alphabet $S'$. Then the group defined by $\Gamma'$ over the alphabet $S'$ is isomorphic to $G(\Gamma)$ since the changes we made correspond to Tietze transformations. If $\Gamma'$ is a forest, then the group defined by $\Gamma$ is the free group on $S'$. Assume this is not the case. Then $\Gamma'$ has girth at least $7$.

Assume $G(\Gamma)$ is cyclic and hence abelian, and let $a,b\in S'$ with $a\neq b$. Then any diagram for the word $aba^{-1}b^{-1}$ over $\Gamma'$ has more than one face. Since there are only 4 boundary edges and the girth of $\Gamma'$ is at least 7, any face has at least 3 interior arcs. Therefore, the first curvature formula is violated. 

We deduce that $G(\Gamma)$ is infinite cyclic if and only if $|S'|= 1$. Note that in this case, $\Gamma'$ is a forest, and the free rank of the free product of the fundamental groups of the connected components of $\Gamma$ is $|S|-1$.
\end{remark}

While, by Example \ref{ex:cayley}, the above theorem cannot hold for arbitrary $Gr(7)$ groups, our method of proof yields the following special case:

\begin{thm}\label{thm:ess-free}
 Let $\Gamma$ be a $Gr(7)$-labelled graph with infinitely many pairwise non-isomorphic connected components with non-trivial fundamental groups. Then $G(\Gamma)$ contains a non-abelian free subgroup.
\end{thm}

\begin{cor}\label{cor:classical-free}
 Let $G$ be a group with an infinite classical $C(7)$ presentation. Then $G$ contains a non-abelian free subgroup.
\end{cor}

\begin{proof}[Proof of Theorem \ref{thm:free} in the finitely presented case]
If $\Gamma$ has finitely \linebreak many distinct simple cycles, Theorem \ref{thm:linear_isoperimetry} and Theorem \ref{thm:asphericity} imply that the group is hyperbolic and torsion-free. It is well-known (cf. \cite[$5.3.\mathrm{C}_1$]{Grhyp}) that non-elementary hyperbolic groups contain non-abelian free subgroups. Any torsion-free elementary hyperbolic group is either trivial or infinite cyclic. Therefore, the claim holds in this case.
\end{proof}

We now make preliminary observations and definitions for the proof of Theorem \ref{thm:free} in the remaining case. The proof of Theorem \ref{thm:ess-free} will be deduced from this proof at the end of the section.

Let $\Gamma$ be a labelled graph. Suppose there is an edge $e$ that is not a piece. Then the label $s$ of $e$ appears only once on the graph. If $e$ lies on a simple cycle, then removing $e$ from $\Gamma$ as well as $s$ from $S$ corresponds to a Tietze transformation and yields an isomorphic group. If $e$ does not lie on a simple cycle, consider the presentation $\langle S|R\rangle$ for $G(\Gamma)$,  where $R$ is the set of words read on simple cycles in $\Gamma$. Then $s$ appears in no word of $R$ and therefore generates a free factor. In this case, the claim of Theorem \ref{thm:free} holds. Therefore, it is no restriction to assume that all edges in $\Gamma$ are pieces, and we therefore do so until our proof is concluded.

Assume that all edges on the labelled graph $\Gamma$ are pieces. Let $x,y$ be vertices lying in a connected component of $\Gamma$. The \emph{piece distance} $d_p(x,y)$ is the least number of pieces whose concatenation is a path from $x$ to $y$. We observe that the restriction of $d_p$ to a connected component of $\Gamma$ is a metric.

\begin{lem}\label{lem:free1} Let $\gamma$ be a simple cycle on a $C(n)$-labelled graph $\Gamma$. Let $x$ be a vertex on $\gamma$. If $n$ is even, there exists a vertex $y$ on $\gamma$ such that $d_p(x,y)=\frac{n}{2}$. If $n$ is odd, one of the following holds:
 \begin{itemize}
  \item There is a vertex $y$ on $\gamma$ such that $d_p(x,y)=\frac{n+1}{2}$, or
  \item there are vertices $y$ and $z$ on $\gamma$ such that $d_p(x,y)=d_p(x,z)=\frac{n-1}{2}$ and $d(y,z)=1$. Any path starting at $x$ and traversing both $y$ and $z$ consists of at least $\frac{n+1}{2}$ pieces. We call the edge connecting $y$ and $z$ the \emph{opposite} edge to $x$.
 \end{itemize}
\end{lem}

\begin{proof}
We observe: For any $x,y$ in a connected component of $\Gamma$ with $d_p(x,y)=d<\frac{n}{2}$, there is a unique simple path from $x$ to $y$ consisting of $d$ pieces. 

Now assume that the simple cycle $\gamma$ is based at $x$. Let $\gamma^+$ be the longest initial subpath that can be made up of less than $\frac{n}{2}$ pieces, and let $\gamma^-$ be the longest terminal subpath with this property. Then, by our observation, $\gamma^+$ and $\gamma^-$ have only the vertex $x$ in common. We may write $\gamma=\gamma^+\delta\gamma^-$ for some nonempty simple path $\delta.$

Assume that $n$ is even. Then $\delta$ consists of at least 2 pieces and hence at least two edges. Any interior vertex $y$ of $\delta$ satisfies $d_p(x,y)\geq \frac{n}{2}$.

Assume that $n$ is odd. If $\delta$ consists of more than one edge, the argument above applies, and the first claim holds. If $\delta$ consists of exactly one edge, the second claim holds. 
\end{proof}

We now recall a basic fact about fundamental groups of graphs: Let $\Gamma$ be a connected graph and $T$ a spanning tree of $\Gamma$. Let $E$ denote the set of edges in $\Gamma$ but not in $T$. Then there is a bijective map from $E$ to a free generating set of $\pi_1(\Gamma,\nu)$ ($\nu$ arbitrary in $\Gamma$). In particular, $\pi_1(\Gamma,\nu)$ is not finitely generated if and only if $E$ is infinite. (For a proof, see \cite[Section 6]{KM}.)

\begin{lem}\label{lem:free2} Let $\Gamma$ be a graph of bounded vertex-degree containing infinitely many pairwise distinct simple cycles. Then $\Gamma$ contains infinitely many pairwise vertex-disjoint simple cycles.
\end{lem}
\begin{proof} First assume that $\Gamma$ has infinitely many connected components that have nontrivial fundamental group. Then the claim is obvious. Therefore, we will assume that $\Gamma$ is connected.

We prove by contradiction. Let $C$ be a maximal set of pairwise disjoint simple cycles on $\Gamma$ and assume that $C$ is finite. Consider a spanning tree $T$ of $\Gamma$. Then there are infinitely many edges in $\Gamma$ that do not lie in $T$, and therefore, since $\Gamma$ is of bounded degree, $T$ is infinite. The forest $T-C$ (i. e. $T$ with all vertices and edges of elements of $C$ removed) has finitely many connected components. Moreover, there are infinitely many edges in $\Gamma$ that do not meet a vertex in $C$ and that do not lie in $T$. Thus one of the following must occur:
\begin{itemize}
 \item There exists a component $K$ of $T-C$ such that there is an edge in $\Gamma$ but not in $T$ and disjoint from $C$ that connects two vertices of $K$, or
 \item there exist two components $K$, $K'$ of $T-C$ such that there are two edges in $\Gamma$ but not in $T$ and disjoint from $C$ connecting vertices of $K$ and $K'$.
\end{itemize}
In both cases, we obtain a simple cycle that is vertex-disjoint from all elements of $C$. This contradicts the maximality of $C$.
\end{proof}

\begin{proof}[Proof of Theorem \ref{thm:free} in the infinitely presented case] 
We assume that the \\ free product of the fundamental groups of the components of $\Gamma$ is not finitely generated. The finitely generated case has been remarked upon.

First assume there are four pairwise vertex-disjoint simple cycles $\gamma_i$, $i\in\{1,2,3,4\}$, such that for each $i$, there are vertices $x_i$ and $y_i$ in $\gamma_i$ with $d_p(x_i,y_i)=4$ . For each $i$, let $w_i$ be the label of a path from $x_i$ to $y_i$. The method of proof below indicates how to show that $\alpha:=w_1w_2$ and $\beta:=w_3w_4$ generate a rank 2 free subgroup of $G(\Gamma)$.

Now assume the first assumption does not hold. We deduce from Lemmas \ref{lem:free1} and \ref{lem:free2} that there are infinitely many pairwise vertex-disjoint simple cycles on $\Gamma$ on each of which the maximal piece distance between two points is 3. Moreover, for any vertex in one of these cycles $\gamma$, there is an opposite edge on $\gamma$. Let $x_a,x_a',x_b,x_b'$ be vertices with opposite edges $a,a',b,b'$ in pairwise disjoint cycles. Since we are considering an infinite set of cycles labelled with a finite set of letters, we may assume that $\ell(a)=\ell(a')$ and $\ell(b)=\ell(b')$.

We add $\alpha,\beta$ to the alphabet $S$ and add the following relations to the presentation:
\begin{itemize}
\item $R_\alpha:=\{\alpha \red(\ell(p)^{-1}\ell(p'))|p:\iota(a)\to x_a,p':\iota(a') \to x_a'\},$
\item $R_\beta:=\{\beta  \red(\ell(p)^{-1}\ell(p'))|p:\iota(b)\to x_b,p':\iota(b') \to x_b'\},$
\end{itemize}
where $\red(-)$ denotes the free reduction, $p:w\to x$ denotes a path from $w$ to $x$, and $\iota(e)$ denotes the initial vertex of an oriented edge $e$. We call a face in a diagram bearing a relator in $R_\alpha\cup R_\beta$ \emph{special}. Let $R$ denote the set of labels of all nontrivial cycles on $\Gamma$. Then
$$G(\Gamma)=\langle S|R\rangle\cong \langle S\cup\{\alpha,\beta\}|R\cup R_\alpha\cup R_\beta\rangle,$$
since we only applied the Tietze transformations of adding generators. (Note that $\ell(p_1)=\ell(p_2)$ in $G$ for any two paths $p_1,p_2$ from vertices $x$ to $y$.) We show that $\alpha$ and $\beta$ generate a free subgroup of rank 2 in $G(\Gamma)$.

Let $D$ be a nontrivial diagram for a cyclically reduced word $\omega$ in $\alpha,\beta$ over the extended presentation. First, we can assume that no edge labelled $\alpha$ or $\beta$ is an interior edge: If there is such an interior edge, it can only occur when two special faces are glued together in such a way that the boundary read around these two faces can be obtained by gluing together two cycles lifting to $\Gamma$. Thus we can replace the two faces containing $\alpha$ (or $\beta$) by two that do not. 

Second, we choose $D$ to be of minimal area with this property and without interior spurs. Then the proofs of Lemma \ref{lem:graphical_basic} and Corollary \ref{cor:graphical_basic} yield that no interior face of $D$ contains an edge originating from $\Gamma$. (Here we associate to the interior edges of special faces the lifts that arise from the fact that they are labels of paths in $\Gamma$.) Hence any arc in the intersection of two interior faces is a piece.

We first cover the case that no interior edge originates from $\Gamma$. We construct from $D$ a diagram $D'$ that violates the first curvature formula by replacing each special face by two faces over the original presentation and performing some foldings. For the sake of brevity we describe the process for one special face $F$ whose label lies in $R_\alpha$:

The subpath of $\partial F$ that lies in the interior of $D$ comes from two reduced paths $p:\iota(a)\to x_a$ and $p':\iota(a')\to x_a'$, where a free reduction may have occurred corresponding to cutting off initial subpaths of $p$ and $p'$ that bear the same labels. 

We replace $F$ by two faces $f$ and $f'$ constructed as follows: $\partial f=\ell (p) \ell (q)$ and $\partial f'=\ell (p') \ell (q')$, where $q$ (resp. $q'$) is a path in $\Gamma$ such that the concatenation $pq$ (resp. $p'q'$) is a closed reduced path in $\Gamma$. The two faces are glued together at the vertex lifting to the initial vertex of $p$ resp. $p'$, and edges of $f$ and $f'$ are folded together starting at this vertex in correspondence to the free reduction of $\ell(p)^{-1}\ell(p')$. Then we glue these two faces into the diagram instead of $F$. Note that interior edges arising from gluing $f$ and $f'$ together do not originate from $\Gamma$. Hence any interior arc in $\partial f$ or $\partial f'$ is a piece. Now there are three cases to consider:
\begin{itemize}
\item[(A)] Suppose neither $\ell(p)$ nor $\ell(p')$ start with $\ell(a)$. Then we can choose $q$ and $q'$ such that $\ell(q)$ and $\ell(q')$ terminate with $\ell(a)^{-1}$. Thus we can fold $f$ and $f'$ together along the corresponding edge as well. This edge does not lift to $\Gamma$. Then the interior paths of $f$ and $f'$ consist of at least 4 pieces by Lemma \ref{lem:free1}.
\item[(B)] Suppose both start with $a$. Then $p$ and $p'$ consist of at least 4 pieces by Lemma \ref{lem:free1}.
\item[(C)] Suppose $\ell(p)$ starts with $\ell(a)$, and $\ell(p')$ does not (or the symmetric case). Then $f$ has interior degree at least 4, $f'$ has interior degree at least 3, and there is no folding between $f$ and $f'$. We can, however just fold together the two exterior edges of $f$ and $f'$ incident at the vertex where the two faces intersect (disregarding the label) to increase the interior degrees of $f$ and $f'$.
\end{itemize}

For all cases we obtain that after these choices and additional foldings, the constructed boundary faces $f$ and $f'$ have interior degree at least 4. Any arc in the intersection of $f$ or $f'$ with an interior face is a piece. Hence, after we performing this procedure for all special faces, we obtain a $(3,7)$-diagram violating the first curvature formula, which is a contradiction.

We are left with the case that $D$ does contain interior edges originating from $\Gamma$. As above, we can construct from $D$ a diagram $D'$ over $\Gamma$ (in particular making the choices in case (A)), but without performing the foldings in case (A) or case (C).  By our assumptions, an interior edge $e$ originating from $\Gamma$ lies in the the boundary of one or two special faces, and therefore it lies in the intersection of boundary faces $f_1$ and $f_2$ of $D'$, where possibly $f_1=f_2$. 

If $f_1=f_2$, then, after forgetting vertices of degree 2 in $D'$, $f_1$ encloses some nontrivial $[3,6]$-subdiagram $\delta$ with at most one boundary vertex of degree (in $\delta$) less than 3, contradicting the second curvature formula. The second curvature formula also yields that any $[3,6]$-diagram with at most two boundary vertices of degree less than 3 is an arc, whence $f_1\cap f_2$ is an arc. 

Since the boundary word of $D$ is cyclically reduced, there can be no vertex in $\partial D'\cap f_1\cap f_2$ whose lifts via $f_1$ and $f_2$ coincide. 
Therefore, $\partial D'\cap f_1\cap f_2$ is empty, and there are nontrivial subdiagrams on both sides of $f_1\cup f_2$. We call two faces of $D'$ containing an edge originating from $\Gamma$ a \emph{bottle neck}.

If there exists a bottle neck, there also exists an \emph{extremal} bottle neck $f_1\cup f_2$, i.e. one such that in one component $C$ of $D'\setminus(f_1\cup f_2)$ there are no more bottle necks. This means in the diagram $\Delta:=C\cup f_1\cup f_2$, all interior edges originating from $\Gamma$ lie in $f_1\cap f_2$.

Now consider the diagram $\Delta$. We can remove all edges originating from $\Gamma$ to merge $f_1$ and $f_2$ to one face $f$. The resulting diagram has no more interior edges originating from $\Gamma$. Thus we can apply our technique from above: Performing the foldings indicated in case (A) and case (C) above yields a $(3,7)$-diagram where all but one boundary face (namely $f$) have interior degree at least 4. This contradicts the first curvature formula.
\end{proof}

\begin{proof}[Proof of Theorem \ref{thm:ess-free}]
This proof is obtained precisely as above by making the following obervations and adjustments:
\begin{itemize}
 \item Since the alphabet $S$ is finite and $\Gamma$ has infinitely many components, we can assume that on all but finitely many components, all edges are essential pieces.
 \item On these infinitely many components, Lemma \ref{lem:free1} can be directly adapted to the $Gr(7)$ case.
 \item Instead of choosing vertex-disjoint cycles, we choose cycles $\gamma_i$ such that that are essentially vertex-disjoint, i. e. $\gamma_i$ and $\phi(\gamma_j)$ are vertex-disjoint for all labelled automorphisms $\phi$ of $\Gamma$ if $i\neq j$. This can be be done due to our assumptions on $\Gamma$.
 \item In the proof, we replace all occurrences of piece, originating from, coincide, etc. by {essential} piece, {essentially} originating from, {essentially} coincide, etc.
\end{itemize}
\end{proof}

%% file: 5coarse_embedding2.tex
\section{Coarse embedding of the graph}

We now consider the question of how $\Gamma$ embeds into $G(\Gamma)$. Similar arguments have been used by Ollivier \cite{Oll} to show that a connected $C'(\frac{1}{6})$-labelled graph embeds isometrically into $G(\Gamma)$. 

\vspace{8pt}

Let $\Gamma$ be a connected, labelled graph. Let $v$ be a vertex in $\Gamma$ and $g\in G(\Gamma)$. There is a unique map of labelled graphs $\iota:\Gamma\to\Cay(G(\Gamma),S)$ satisfying $\iota(v)=g$. We prove a claim of Gromov \cite[Theorem 2.3]{Gr}:

\begin{lem}\label{lem:embedding_injective}
 Let $\Gamma$ be a connected component of a $Gr(6)$-labelled graph. Then the map $\iota:\Gamma\to\Cay(G(\Gamma),S)$ is injective.
\end{lem}
\begin{proof}
Let $x\neq y\in \Gamma$, and assume $\iota(x)=\iota(y)$. Let $p$ be a path from $x$ to $y$ on $\Gamma$ such that the area of $w=\ell(p)$ is minimal among all paths from $x$ to $y$. We may assume that $w$ is reduced. Since $x\neq y$, this area is greater than zero. Let $D$ be a minimal diagram for $w$ over $\Gamma$. For each boundary edge of $D$ bounding a face, there is a lift to $\Gamma$ via $p$ and one via the face it lies on. Suppose there were an edge $e$ bounding a face $f$ for which these two lifts were essentially equal. Then we could remove $e$ and the interior of $f$. The resulting diagram would have a boundary word read on a path connecting $x$ and $y$ of lesser area than $D$, which is a contradiction. Thus every arc of $D$ bounding a face is a piece. By construction, $D$ has at most one spur (connecting the base vertex to the rest of the diagram) and hence at most one vertex of degree one. Now forgetting vertices of degree two leaves a $[3,6]$-diagram violating the second curvature formula, which 
is a contradiction. 
\end{proof}

\begin{cor}
 For any $C(6)$-labelled graph $\Gamma$ with a connected component having more than one vertex, the group $G(\Gamma)$ is infinite.
\end{cor}
\begin{proof}
 This follows since $G(\Gamma)$ is nontrivial and, by Theorem \ref{thm:asphericity}, torsion-free.
\end{proof}

Let $(\Gamma_n)_{n\in\N}$ be a sequence of connected finite graphs and let $\Gamma:=\bigsqcup_{n\in\N}\Gamma_n$ be their disjoint union. We endow $\Gamma$ with a metric that coincides with the graph metric on each connected component such that $d(\Gamma_ {a_n},\Gamma_{b_n})\to\infty$ as $a_n+b_n\to \infty$ assuming $a_n\neq b_n$ for almost all $n$. (For example, set $d(x,y)=\diam(X_m)+\diam(X_n)+m+n$ if $x\in X_m$, $y \in X_n$ and $m\neq n$.) We call the resulting metric space the \emph{coarse union} of the $\Gamma_n$. 

If $\Cay(G(\Gamma),S)$ is infinite, it contains an infinite geodesic ray. Then we can map $\Gamma$ into $\Cay(G(\Gamma),S)$ via a map of labelled graphs $\iota$ by lining the $\Gamma_n$'s up on this geodesic ray such that, for all sequences $a_n,b_n$ we have $d(\iota(\Gamma_{a_n}),\iota(\Gamma_{b_n}))\to \infty$ if and only if $d(\Gamma_{a_n},\Gamma_{b_n})\to \infty$. We claim that such a map $\iota$ is a \emph{coarse embedding}, i.e.\ it satisfies for every sequence of pairs of points $(x_n,y_n)_{n\in\N}$ in $\Gamma\times \Gamma$:
$$d(x_n,y_n)\to \infty \Leftrightarrow d(\iota(x_n),\iota(y_n))\to\infty.$$

\begin{thm}\label{thm:coarse_embedding} Let $(\Gamma_n)_{n\in\N}$ be a sequence of connected finite graphs such that $\Gamma:=\bigsqcup_{n\in\N}\Gamma_n$ is $Gr(6)$-labelled and such that $|\Gamma_n|$ is unbounded. Then the coarse union $\Gamma$ embeds coarsely into $\Cay(G(\Gamma),S)$.
\end{thm}

\begin{proof}
First note that the assumption implies that $\Cay(G(\Gamma),S)$ is infinite, since by Lemma \ref{lem:embedding_injective}, each $\Gamma_n$ maps injectively into $\Cay(G,S)$.

Let $(x_n,y_n)_{n\in\N}$ be a sequence of pairs of points such that $d(x_n,y_n)\to\infty$. We claim: $d(\iota(x_n),\iota(y_n))\to\infty$. By the above construction of $\iota$, it is sufficient to consider the case where for all $n$, $x_n$ and $y_n$ lie in the same connected component. Suppose our claim is false. Then $(x_n,y_n)$ has a subsequence $(x_n',y_n')$ such that $d(\iota(x_n'),\iota(y_n'))$ is bounded. Hence there is a subsequence $(x_n'',y_n'')$ such that for all $n$, the labels of paths from $x_n''$ to $y_n''$ define the same element $w$ of $G$. Since all $\Gamma_n$ are bounded, we also assume that for $n\neq m$, the graphs containing $\{x_n'',y_n''\}$ and $\{x_m'',y_m''\}$ are distinct and non-isomorphic. 

Let $n\in\N$ and choose paths $p_n,q_n$ from $x_1''$ to $y_1''$, respectively from $x_n''$ to $y_n''$, such that the minimal area of a diagram for the word $\ell(p_n)\ell(q_n)^{-1}$ over $\Gamma$ is minimal. We claim that the area of such a minimal diagram $D_n$ is zero. We may assume that the paths are reduced. We can argue as in Lemma \ref{lem:embedding_injective} that every arc bounding a face in $D_n$ is a piece. Now forgetting vertices of degree two yields a $[3,6]$-diagram with at most two vertices of degree one. By the second curvature formula, such a diagram has exactly two vertices and hence no faces. Thus $D_n$ has no face, $\ell(p_n)=\ell(q_n)$, and $p_n$ and $q_n$ are essential pieces and hence simple paths. 

Let $X$ be the connected component of $\Gamma$ containing $x_1''$. Then $|q_n|=|p_n|$ is bounded from above by the maximal length of a simple path on the finite graph $X$. Since $n$ was arbitrary, this contradicts the assumption $d(x_n'',y_n'')\to\infty$.

Obviously, $d(\iota(x),\iota(y))\leq d(x,y)$ for any $x,y$ in a connected component of $\Gamma$. Hence $d(\iota(x_n),\iota(y_n))\to\infty$ implies $d(x_n,y_n)\to\infty$.
\end{proof}

\begin{example}
 Let $p\in\N$. We construct a sequence of finite, connected labelled graphs $(\Gamma_n)_{n\in\N}$ such that their disjoint union $\Gamma$ has a $C(p)$-labelling and such that any map of labelled graphs $\iota:\Gamma\to\Cay(G(\Gamma),S)$ is \emph{not} a quasi-isometric embedding. Let $S=S_1\cup S_2\cup\{a,b\}$, where $S_1,S_2$ and $\{a,b\}$ are pairwise disjoint and $|S_1|>1$. Let $(w_n)_{n\in\N}$ be a sequence of pairwise distinct, reduced words in the free monoid $M(S_1)$ on $S_1\cup S_1^{-1}$ such that $|w_n|=O(\log n)$, and let $(v_n)_{n\in\N}$ a sequence of pairwise distinct, reduced words in $M(S_2)$. For each $n$, let $\Gamma_n$ be the graph given in Figure \ref{figure:example_embedding}:

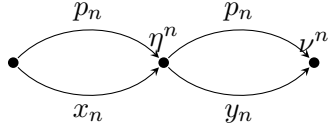
\begin{figure}[ht]
\begin{tikzpicture}[>=stealth,shorten <=2.5pt, shorten >=2.5pt]
\coordinate (A) at (0,0);
\coordinate [label=above:\small $\eta^n$] (B) at (2,0);
\coordinate [label=above:\small $\nu^n$] (C) at (4,0);
\fill (A) circle (2pt);
\fill (B) circle (2pt);
\fill (C) circle (2pt);
\draw[->] (A) to [out=45,in=135] node [above] {\small $p_n$} (B);
\draw[->] (A) to [out=-45,in=-135] node [below] {\small $x_n$} (B);
\draw[->] (B) to [out=45,in=135] node [above] {\small $p_n$} (C);
\draw[->] (B) to [out=-45,in=-135] node [below] {\small $y_n$} (C);
\end{tikzpicture}
\caption{The graph $\Gamma_n$.}
\label{figure:example_embedding}
\end{figure}

\noindent Here $p_n,x_n,y_n$ are words in $S$ given as follows:

\begin{itemize}
 \item $p_n=b^{f(n)}$, where $f(n)$ is some function $\N\to\N$ such that $\log n=o(f(n))$,
 \item $x_n=w_{(p-1)n+1}a w_{(p-1)n+2}a\ldots aw_{pn}a$,
 \item $y_n=v_{(p-1)n+1}a v_{(p-1)n+2}a\ldots av_{pn}a^{f(n)}$.
\end{itemize}

The labels $\eta^n$ and $\nu^n$ are vertex-labels for reference. The reader can easily check that the resulting graph $\Gamma$ has a $C(p)$-labelling by considering how many non-consecutive instances of $a$ can occur in a piece. By construction $|p_n|\leq |y_n|$ and hence $d(\eta^n,\nu^n)=|p_n|$. Since $p_n$ and $x_n$ are equal in $G(\Gamma)$, we have $d(\iota(\eta^n),\iota(\nu^n))\leq |x_n|=o(d(\eta^n,\nu^n))$. Therefore, $\iota$ cannot be a quasi-isometric embedding.

\end{example}

%% file: 6lacunary_hyperbolicity.tex
\section{Lacunary hyperbolicity}

We now consider an infinite sequence of finite graphs $(\Gamma_n)_{n\in\N}$ such that their disjoint union is $Gr(7)$-labelled or $Gr'(\frac{1}{6})$-labelled. The resulting group is a limit of hyperbolic groups. In the view of \cite{OOS} it is natural to ask whether these groups are \emph{lacunary hyperbolic}. We give some preliminary definitions:

\begin{defi}[Ultrafilter] An \emph{ultrafilter} is a finitely additive map $\omega:2^\N\to\{0,1\}$ such that $\omega(\N)=1$. An ultrafilter $\omega$ is called \emph{non-principal} if $\omega(F)=0$ for all finite subsets $F$ of $\N$. Let $f:\N\to \R$ be a sequence. Then for $x\in \R$ we say $x=\lim_n^{\omega}f(n)$ if $\forall \epsilon >0: \omega(f^{-1}([x-\epsilon,x+\epsilon]))=1$.
\end{defi}
It is a fact that given an ultrafilter $\omega$, any bounded sequence $f:\N\to\R$ has a limit with respect to $\omega$.

\begin{defi}[Asymptotic cone] Let $\omega$ be a non-principal ultrafilter and $(d_n)_{n\in\N}$ a sequence of real numbers such that $d_n\to\infty$ as $n\to\infty$. The sequence $d_n$ is called \emph{scaling sequence}. Let $G$ be a group generated by a finite set $S$. Let $G^\N$ denote the space of sequences of elements of $G$, and $X:=\{(x_n)\in G^\N|d_S(1,x_n)=O(d_n)\}$, where $d_S$ denotes the distance in $\Cay(G,S)$. We define a pseudo-metric $d$ on $X$ by setting $d((x_n),(y_n))=\lim_n^\omega \frac{d(x_n,y_n)}{d_n}$. An equivalence relation on $X$ is induced by $(x_n)\sim (y_n):\Leftrightarrow d(x_n,y_n)=0$. The \emph{asymptotic cone} of $G$ with respect to $S$, $\omega$ and $(d_n)$ is defined as $X/\sim$ with the metric induced by $d$. We denote it by $\Con^w(G,d_n)$.
\end{defi}

Let $G$ be a group generated by a finite set $S$, and let $\pi:G\to H$ be a group homomorphism. Then the \emph{injectivity radius} of $\pi$, denoted $r_S(\pi)$, is the largest $r\in\R$ such that $\pi$ restricted to the open ball of radius $r$ at $1$ in $\Cay(G,S)$ is injective. An $\R$-tree is a 0-hyperbolic space (see Definition \ref{defi:hyperbolic_group}). In particular, an $\R$-tree contains no nontrivial embedded cycles. We recall the definition and a characterization of lacunary hyperbolicity from \cite[Section 3.1]{OOS}:

\begin{defi}[Lacunary hyperbolic group]
 A finitely generated group $G$ is called \emph{lacunary hyperbolic} if the following equivalent conditions hold:
\begin{itemize}
 \item Some asymptotic cone of $G$ is an $\R$-tree.
 \item $G$ is the direct limit of finitely generated groups and epimorphisms
 $$G_1\xrightarrow{\alpha_1}G_2\xrightarrow{\alpha_2}\ldots$$
 such that $G_i$ is generated by a finite set $S_i$, $\alpha_i(S_i)=S_{i+1}$ and each $G_i$ is $\delta_i$-hyperbolic, where $\delta_i=o(r_{S_i}(\alpha_i))$.
\end{itemize}
\end{defi}

\subsection{The $Gr(7)$ case}

Since the defining condition for lacunary hyperbolicity has a metric nature, the $Gr(7)$ condition does not yield optimal results. We are able to prove, however:

\begin{prop}\label{prop:7_lacunary}
 Let $(\Gamma_n)_{n\in\N}$ be a sequence of finite, connected labelled graphs such that their disjoint union is $Gr(7)$-labelled. Then there exists an infinite subsequence of graphs $(\Gamma_{k_n})_{n\in\N}$ such that $G(\bigsqcup_{n\in\N}\Gamma_{k_n})$ is lacunary hyperbolic.
\end{prop}

\begin{lem}\label{lem:c7_injectivity_infinity}
Let $\Gamma$ be a finite labelled graph and $(\Gamma_n)_{n\in\N}$ a sequence of connected, finite labelled graphs such that $\Gamma':=\Gamma\sqcup\bigsqcup_{n\in\N}\Gamma_n$ is $Gr(7)$-labelled and such that for $n\neq n'$, the labelled graphs $\Gamma_n$ and $\Gamma_{n'}$ are non-isomorphic. Then the injectivity radii $\rho_n$ of the projections $\pi_n:G(\Gamma)\to G(\Gamma\sqcup\Gamma_n)$ induced by the identity on $S$ tend to infinity. 
\end{lem}

\begin{proof}
Suppose this is false. Then there exists an infinite sequence $(k_n)_{n\in\N}$ and a word $w$ such that $w$ is nontrivial in $G(\Gamma)$ and such that for all $n\in\N$, $w$ is trivial in $G(\Gamma\sqcup \Gamma_{k_n})$ and $|w|=\rho_{k_n}$. For each $n\in\N$, let $D_n$ be a diagram for $w$ over $\Gamma\sqcup\Gamma_{k_n}$ such that the labels of all faces are read on simple cycles and such that no interior edge originates from the graph. By construction, each $D_n$ contains at least one face from $\Gamma_{k_n}$. For each $n\in\N$, let $D_n'$ be an inclusion-minimal subdiagram of $D_n$ containing all faces that lift to $\Gamma_{k_n}$. Then all boundary faces of $D_n'$ lift to $\Gamma_{k_n}$. 

The first curvature formula implies that for each $n$, there is at least one boundary face $f_n$ of $D_n'$ such that one arc $p_n$ of $f_n\cap\partial D_n'$ is not an essential piece in $\Gamma'$. This implies that all $p_n$ bear distinct labels and hence $|p_n|\to \infty$. Note that $p_n$ is a subpath of $\partial D_n$ and/or the boundaries of some faces of $D_n$ that lift to $\Gamma$. By Theorem \ref{thm:linear_isoperimetry}, the length of any such path is bounded by $|w|+8|w| (V(\Gamma)+1)$, where $V(\Gamma)$ is the number of vertices of $\Gamma$ (Hence $V(\Gamma)+1$ is an upper bound for the length of a simple closed path on $\Gamma$). This is a contradiction.
\end{proof}

\begin{proof}[Proof of the proposition]
If all but finitely many $\Gamma_n$ are isomorphic, we can consider $G(\Gamma)$ as given by a finite $Gr(7)$-labelled graph. In that case, $G(\Gamma)$ is hyperbolic and hence lacunary hyperbolic. 

Therefore, we can assume that for $n\neq n'$, $\Gamma_n$ and $\Gamma_{n'}$ are non-isomorphic. We choose the subsequence recursively: Let $k_0=1$, and let $k_1,\ldots,k_N$ be chosen. Set $\Gamma^N:=\sqcup_{i=1}^N \Gamma_{k_i}$. Then, by Theorem \ref{thm:linear_isoperimetry}, $G(\Gamma^N)$ is $\delta_N$-hyperbolic for some $\delta_N>0$. By Lemma \ref{lem:c7_injectivity_infinity}, the injectivity radii $\rho_n$ of the maps $G(\Gamma^N)\to G(\Gamma^N\sqcup \Gamma_n)$ induced by the identity on $S$ tend to infinity. Hence we may choose $k_{N+1}$ such that $k_{N+1}>\max\{k_1,\ldots,k_N\}$ and such that $\rho_{k_{N+1}}>N\delta_N$. The resulting limit group $G(\sqcup_{i\in\N}\Gamma_{k_i})$ is lacunary hyperbolic.
\end{proof}

\subsection{The $Gr'(\frac{1}{6})$ case}

Finitely presented $C'(\frac{1}{6})$ groups have been investigated in [Oll]. (In fact, Ollivier uses a stronger variant of the $C'(\frac{1}{6})$ condition.) His proofs of the facts we will use generalize to our $Gr'(\frac{1}{6})$ condition. The question, when classical $C'(\frac{1}{6})$ groups are lacunary hyperbolic has been solved in \cite[Proposition 3.12]{OOS}. We follow their arguments to extend this result:

\begin{prop}\label{prop:1/6_lacunary} Let $(\Gamma_n)_{n\in\N}$ be a sequence of finite, connected graphs such that
 \begin{itemize}
  \item $\Gamma:=\sqcup_{n\in\N}\Gamma_n$ has a $Gr'(\frac{1}{6})$-labelling and
  \item $\Delta(\Gamma_n)=O(g(\Gamma_n))$,
 \end{itemize}
where $\Delta(\Gamma_n)$ denotes the diameter and $g(\Gamma_n)$ denotes the girth of each graph. (If $\Gamma_n$ is a tree, set $g(\Gamma_n)=0$.)
Then $G(\Gamma)$ is lacunary hyperbolic if and only if the set of girths $L:=\{g(\Gamma_n)|n\in\N\}$ is sparse, i.e.\ for all $K\in \R^+$ there exists $a \in \R^+$ such that $[a,aK]\cap L=\emptyset$.
\end{prop}

We first give a precise estimate for the hyperbolicity constant of a group given by a finite $Gr'(\frac{1}{6})$-labelled graph. We observe as in Section \ref{section:preliminaries}:

\begin{lem}\label{lem:metric_diagrams} Let $\Gamma$ be a $Gr'(\frac{1}{6})$-labelled graph. Let $w\in M(S)$ satisfying $w=1$ in $G(\Gamma)$ and let $D$ be a minimal diagram over $\Gamma$ for $w$. Let $f$ be a face of $D$. Then $f$ is simply connected. Any interior arc of $\partial f$ has length less than $\frac{1}{6}|\partial f|$.
\end{lem} 

Note that for a singular disk diagram $D$ over a group presentation $G=\langle S|R\rangle$, the 1-skeleton of $D$ maps into $\Cay(G,S)$. Singular disk diagrams with the property of Lemma \ref{lem:metric_diagrams}  whose boundaries are geodesic triangles in the Cayley graph have been classified in \cite[Section 3.4]{Str}. This classification yields:

\begin{lem}\label{lem:metric_delta} Let $\Gamma$ be a finite $Gr'(\frac{1}{6})$-labelled graph, and let $\Delta$ be the maximum of the diameters of its connected components. Then $G(\Gamma)$ is $2\Delta$-hyperbolic.
\end{lem}

We will use the following two facts \cite[Lemma 13 and Theorem 1]{Oll} proven by Ollivier:
\begin{lem}\label{lem:metric_injectivity}
 Let $\Gamma$ be a $Gr'(\frac{1}{6})$-labelled graph. Let $w\in M(S)$ satisfying $w=1$ in $G(\Gamma)$, and let $D$ be a minimal diagram for $w$ over $\Gamma$. Then for any face $f$, we have $|\partial f|\leq |\partial D|$.
\end{lem}

\begin{thm}\label{thm:metric_embedding} Let $\Gamma$ be a $Gr'(\frac{1}{6})$-labelled graph. Then each connected component of $\Gamma$ embeds isometrically into $\Cay(G(\Gamma),S)$.
\end{thm}

\begin{proof}[Proof of the proposition] Note that given the previous lemmas and theorem, a proof can be deduced from
\cite[Proof of Proposition 3.12]{OOS}. In fact, the first part of our proof uses the same arguments.

\vspace{8pt}

By identifying isomorphic connected components of $\Gamma$, we can assume that for $n\neq n'$, $\Gamma_n$ and $\Gamma_{n'}$ are non-isomorphic. Assume that $L$ is sparse. Then there exists a sequence $(\alpha_n)_{n\in\N}$ of positive real numbers such that for all $n\in \N$, we have $L\cap [\alpha_n,n \alpha_n]=\emptyset$. We may assume $n\alpha _n<\alpha_{n+1}$ for all $n$. Fix $C\in\R^+$ such that for all $n: \Delta(\Gamma_n)\leq C g(\Gamma_n)$. Let 
$$\Gamma^k:=\sqcup_{g(\Gamma_n)<\alpha_k}\Gamma_n.$$
Then for all $k$, the graph $\Gamma^k$ is a finite and $Gr'(\frac{1}{6})$-labelled. For any connected component of $\Gamma^k$, the girth is bounded from above by $\alpha_k$, and hence the diameter is bounded by $C\alpha_k$. Lemma \ref{lem:metric_delta} implies that $G_k:=G(\Gamma^k)$ is $2C\alpha_k$-hyperbolic. Set $\delta_k:=2C\alpha_k$.

The injectivity radius $r_S(\pi_k)$ of the map $G_k\to G_{k+1}$ induced by the identity on $S$ is the length of the shortest word $w$ in $S$ that is trivial in $G_{k+1}$ but not trivial in $G_k$. Hence a minimal diagram $D$ for $w$ over $\Gamma^{k+1}$ has a face $f$ from $\Gamma^{k+1}\setminus \Gamma^k$. Such a face satisfies $|\partial f|\geq k\alpha_k$ by construction. Hence Lemma \ref{lem:metric_injectivity} implies $r_S(\pi_k)=|w|\geq k\alpha_k$. Therefore $\frac{\delta_k}{r_S(\pi_k)}\to 0$ as $k\to \infty$, and $G(\Gamma)$ is lacunary hyperbolic.

\vspace{8pt}

We now prove the converse. Assume that the set of girths $L$ is not sparse. Choose any non-decreasing scaling sequence $(d_n)_{n\in\N}$ tending to infinity and any non-principal ultrafilter $\omega$ on $\N$. We show that $Y:=\Con^\omega(G(\Gamma),d_n)$ is not an $\R$-tree. Theorem \ref{thm:metric_embedding} implies that from each graph $\Gamma_n$, there is a cycle $p_n$ of length $g(\Gamma_n)$ isometrically embedded into $X:=\Cay(G(\Gamma),S)$. Since $L$ is not sparse, there exists $K>0$ such that for all $a >0$ we have $[a,aK]\cap L\neq\emptyset$. Hence for all $n\in\N$ there is $k(n)\in \N$ such that the inequality
$$d_n\leq g(\Gamma_{k(n)})\leq K d_n $$
holds, or in other words $1\leq |p_{k(n)}|/d_n\leq K$.  Since the interval $[1,K]$ is bounded, the sequence $(|p_{k(n)}|/d_n)_{n\in\N}$ converges to some $R\in[1,K]$ with respect to the ultrafilter $\omega$. We may view each cycle $p_{k(n)}$ as a continuous map $p_{k(n)}:\R/\Z\to X$, and we may assume that for all $t$ and for all $\epsilon\in(-\frac{1}{2},\frac{1}{2}]$ we have $d(p_{k(n)}(t),p_{k(n)}(t+\epsilon))=|\epsilon| |p_{k(n)}|$. 

Consider the cycle $\gamma:\R/\Z\to Y, t\mapsto [(p_{k(n)}(t))_{n\in \N}]$. Let $t\neq t'\in \R/\Z$ and $\epsilon\in(-\frac{1}{2},\frac{1}{2}]$ such that $t'=t+\epsilon$. Then for any $n$,
$$d(\gamma(t),\gamma(t'))=\lim_n^\omega \frac{d(p_{k(n)}(t),p_{k(n)}(t+\epsilon))}{d_n}=\lim_n^\omega|\epsilon|\frac{|p_{k(n)}|}{ d_n}=|\epsilon|R>0.$$ Hence $\gamma$ is injective, and thus $Y$ is not an $\R$-tree.
\end{proof}